\documentclass[a4paper, 11pt]{amsart}
\usepackage[utf8]{inputenc}
\usepackage{microtype}
\usepackage{amsmath,amsthm,amsfonts,amssymb}
\usepackage{hyperref,graphicx}
\hypersetup{
    colorlinks=true,
    linkcolor=magenta,
    citecolor=blue,
    urlcolor=cyan,
}
\usepackage[margin=20mm]{geometry}

\newcommand{\E}{\mathbb{E}}
\newcommand{\R}{\mathbb{R}}
\newcommand{\C}{\mathbb{C}}

\newcommand{\Z}{\mathbb{Z}}

\newcommand{\tr}{\text{tr}}
\newcommand{\A}{\mathcal{A}}
\newcommand{\state}{\varphi}
\newcommand{\Ginv}{G^{\langle -1 \rangle}}

\def\bone{\mathbf{1}}

\def\cir{\mathrm{circ}}
\def\semicirc{\mathrm{sc}}
\def\cov{\mathrm{cov}}
\def\bi{\mathbf{i}}
\def\P{\mathcal{P}}

\newtheorem{prop}{Proposition}[section]
\theoremstyle{plain}
\newtheorem{theorem}{Theorem}[section]
\theoremstyle{plain}
\newtheorem{lemma}{Lemma}[section]
\theoremstyle{plain}
\newtheorem{cor}{Corollary}[section]
\theoremstyle{plain}
\newtheorem{definition}{Definition}[section]
\theoremstyle{plain}
\newtheorem{assumption}{Assumption}[section]
\theoremstyle{plain}
\newtheorem{conj}{Conjecture}[section]
\theoremstyle{plain}

\theoremstyle{remark}
\newtheorem{remark}{Remark}[section]
\theoremstyle{remark}

\newenvironment{assumption-alt}[1]
  {\renewcommand{\theassumption}{\ref*{#1}$'$}%
   \begin{assumption}}
  {\end{assumption}}

\title[On $*$-Convergence of Schur-Hadamard Products]{On $*$-Convergence of Schur-Hadamard Products of Independent Nonsymmetric Random Matrices}
\author[S. S. Mukherjee]{Soumendu Sundar Mukherjee}
\address{Interdisciplinary Statistical Research Unit, Indian Statistical Institute, 203 B. T.~Road, Kolkata 700108, India}
\address{Department of Mathematics, National University of Singapore, 10 Lower Kent Ridge Road, Singapore 119076}
\email{soumendu041@gmail.com}
\thanks{The author is supported by an INSPIRE Faculty Fellowship from the Department of Science and Technology, Government of India.}

\begin{document}
\begin{abstract}
    Let $\{x_{\alpha}\}_{\alpha \in \Z}$ and $\{y_{\alpha}\}_{\alpha \in \Z}$ be two independent collections of zero mean, unit variance random variables with uniformly bounded moments of all orders. Consider a nonsymmetric Toeplitz matrix $X_n = ((x_{i - j}))_{1 \le i, j \le n}$ and a Hankel matrix $Y_n = ((y_{i + j}))_{1 \le i, j \le n}$, and let $M_n = X_n \odot Y_n$ be their elementwise/Schur-Hadamard product. In this article, we show that almost surely, $n^{-1/2}M_n$, as an element of the $*$-probability space $(\mathcal{M}_n(\C), \frac{1}{n}\tr)$, converges in $*$-distribution to a circular variable. With i.i.d. Rademacher entries, this construction gives a matrix model for circular variables with only $O(n)$ bits of randomness. We also consider a dependent setup where $\{x_{\alpha}\}$ and $\{y_{\beta}\}$ are independent strongly multiplicative systems (\`{a} la Gaposhkin \cite{gaposhkin1969central}) satisfying an additional \emph{admissibility} condition, and have uniformly bounded moments of all orders---a nontrivial example of such a system being $\{\sqrt{2}\sin(2^n \pi U)\}_{n \in \Z_+}$, where $U \sim \mathrm{Uniform}(0, 1)$. In this case, we show in-expectation and in-probability convergence of the $*$-moments of $n^{-1/2}M_n$ to those of a circular variable. Finally, we generalise our results to Schur-Hadamard products of structured random matrices of the form $X_n = ((x_{L_X(i, j)}))_{1 \le i, j \le n}$ and $Y_n = ((y_{L_Y(i, j)}))_{1 \le i, j \le n}$, under certain assumptions on the \emph{link-functions} $L_X$ and $L_Y$, most notably the injectivity of the map $(i, j) \mapsto (L_X(i, j), L_Y(i, j))$. Based on numerical evidence, we conjecture that the circular law $\mu_{\cir}$, i.e. the uniform measure on the unit disk of $\C$, which is also the Brown measure of a circular variable, is in fact the limiting spectral measure (LSM) of $n^{-1/2}M_n$. If true, this would furnish an interesting example where a random matrix with only $O(n)$ bits of randomness has the circular law as its LSM.
\end{abstract}

\maketitle

\section{Introduction}
Let $A$ be an $n \times n$ complex matrix with eigenvalues $\lambda_1, \ldots, \lambda_n$. The \emph{empirical spectral measure} (ESM) of $A$ is the probability measure given by
\begin{equation}
    \mu_A = \frac{1}{n}\sum_{i = 1}^n \delta_{\lambda_i},
\end{equation}
where $\delta_x$ is the Dirac measure at $x$. The circular law $\mu_{\mathrm{circ}}$ is the uniform measure on the unit disk of $\C$. Consider IID matrices $A_n = ((a_{ij}))$, where $a_{ij}$ are i.i.d. zero mean unit variance complex random variables. The famous circular law theorem \cite{tao2010random} states that the ESMs $\mu_{n^{-1/2}A_n}$ converge weakly almost surely to $\mu_{\cir}$. Weak limits of ESMs are called \emph{limiting spectral measures} (LSMs).

A patterned/structured random matrix is a matrix $A$ whose $(i, j)$-th entry is given by $a_{L(i, j)}$, where $L : \Z_+^2 \rightarrow \Z^d$ is a \emph{link-function} which dictates the pattern, and $\{a_\alpha\}_{\alpha \in \Z^d}$ is a collection of random variables with zero mean, unit variance. The random variables $a_{\alpha}$ are typically assumed to be i.i.d. for showing almost sure weak convergence of ESMs. A few notable examples of patterned matrices are given in Table~\ref{table:srm}.

\begin{table}[!ht]
    \centering
    \caption{Some common structured random matrices.}
    \label{table:srm}
    \begin{tabular}{lll} \hline
        Matrix                 & Link function            & $d$ \\ \hline
        IID                    & $(i, j)$                 & $2$ \\
        Wigner                 & $(i \wedge j, i \vee j)$ & $2$ \\
        nonsymmetric Toeplitz & $i - j$                  & $1$ \\
        Hankel                 & $i + j$                  & $1$ \\ \hline
    \end{tabular}
\end{table}

Existence of almost sure weak limits of the ESMs of various self-adjoint patterned random matrices is known. For example, the Wigner matrix has the \emph{semi-circular} law $\mu_{\semicirc}$ as its LSM, which is a probability measure on $\R$ with density
\begin{equation}
    \mu_{\semicirc}(dt) = \frac{1}{2\pi}\sqrt{4 - t^2} \, \bone_{\{|t| \le 2\}}\, dt.
\end{equation}
Symmetric Toeplitz ($L(i, j) = |i - j|$) and Hankel matrices also have LSMs \cite{bryc2006spectral, hammond2005distribution}; however, their densities are not known explicitly. For a unified treatment of symmetric patterned random matrices, see \cite{bose2018patterned}.

On the other hand, it is not known if nonsymmetric Toeplitz matrices have LSMs. In fact, besides IID matrices and their variants which only require independence of the entries but not identical distribution, weak limits are known for only a handful of models in the nonsymmetric/non-self-adjoint case (see, e.g., \cite{meckes2009some, guionnet2011single, nguyen2015elliptic, basak2018circular}).

In this article, we consider elementwise/Schur-Hadamard products (denoted by the symbol $\odot$) of independent nonsymmetric patterned random matrices. Let $L_X : \Z_+^2 \rightarrow \Z^d$ and $L_Y : \Z_+^2 \rightarrow \Z^{d'}$ be two different link-functions and $\{x_\alpha\}_{\alpha \in \Z^d}$ and $\{y_{\beta}\}_{\beta \in \Z^{d'}}$ two independent collections of random variables with zero mean and unit variance. Let
\[
    X_n = ((x_{L_X(i, j)}))_{1 \le i, j \le n} \text{ and } Y_n = ((y_{L_Y(i, j)}))_{1\le i, j \le n}.
\]
Consider the Schur-Hadamard product of $X_n$ and $Y_n$:
\begin{equation}
   M_n := X_n \odot Y_n = ((X_{n, i, j} Y_{n, i, j})) = ((x_{L_X(i, j)} y_{L_Y(i, j)})).
\end{equation}
Note that by our assumptions on $x_{\alpha}$ and $y_{\alpha}$, the entries of $M_n$ remain zero mean, unit variance.

Schur-Hadamard products of symmetric patterned random matrices were considered in \cite{bose2014bulk} and a main result of that paper says that if $X_n$ is symmetric Toeplitz and $Y_n$ is Hankel, then the LSM of $n^{-1/2} M_n$ exists and is in fact the semi-circular law $\mu_{\semicirc}$.

In this article, we consider the Schur-Hadamard product of nonsymmetric Toeplitz and Hankel matrices, i.e. we take $L_X(i, j) = i - j$ and $L_Y(i, j) = i + j$. Simulations shown in Figure~\ref{fig:toephank} prompt us to make the following conjecture.
 
\begin{conj}\label{conj1}
    Let $L_X(i, j) = i - j$ and $L_Y(i, j) = i + j$. Suppose that $\{x_{\alpha}\}_{\alpha \in \Z}$ and $\{ y_{\alpha}\}_{\alpha \in \Z}$ are independent collections of i.i.d. random variables with zero mean and unit variance. Then the ESM of $n^{-1/2}M_n$ converges weakly almost surely to the circular law $\mu_{\mathrm{circ}}$.
\end{conj}

In fact, one of our main motivations for studying the Schur-Hadamard product model is that Conjecture~\ref{conj1}, if true, would provide an interesting example of a random matrix model with only $O(n)$ bits of randomness with $\mu_{\mathrm{circ}}$ as the LSM. More elaborately, one can take the entries to be independent Rademacher variables, i.e. random signs, as in Figure~\ref{fig:toephank}-(b). All the standard models with $\mu_{\mathrm{circ}}$ as LSM possess $\Omega(n^2)$ bits of randomness (an exception is the recent paper \cite{basak2018circular}, where the authors require $\Omega(n\log^{12}(n))$ bits of randomness).

In this article, we consider the problem from a free-probabilistic perspective. We show in Theorem~\ref{thm:toephank} that $n^{-1/2}M_n$, as an element of the $*$-probability space $(\mathcal{M}_n(L^{\infty, -}(\Omega, \mathbb{P})), \frac{1}{n}\E\tr)$ converges in $*$-distribution to a circular variable. We originally showed this for input random variables which are independent with zero mean, unit variance, and have uniformly bounded moments of all orders. One of the anonymous referees pointed out that the classical central limit theorem holds for a strongly multiplicative system (SMS)\footnote{Multiplicative systems of random variables were first studied by \cite{alexits1961convergence}. For a central limit theorem for SMSs, see, e.g., \cite{gaposhkin1969central}.} of random variables (see Definition~\ref{defn:sms}), and raised the question if we can allow such entries. 

In Theorem~\ref{thm:toephank}, we allow the input random variables $\{x_\alpha\}$ and $\{y_\beta\}$ to be independent SMSs satisfying a certain admissibility condition (see Definition~\ref{defn:asms}) and having uniformly bounded moments of all orders. Independent random variables with zero mean, unit variance, and uniformly bounded moments of all orders are trivial examples of such systems. A nontrivial example is the collection $\{\sqrt{2}\sin(2^n \pi U)\}_{n \in \Z_+}$, where $U \sim \mathrm{Uniform}(0, 1)$. Simulations suggest that the circular law limit is possibly true even under such dependence. See Figure~\ref{fig:toephank}-(c).

When the entries are independent random variables with zero mean, unit variance, and uniformly bounded moments of all orders, we further show that almost surely, $n^{-1/2}M_n$, as an element of the $*$-probability space $(\mathcal{M}_n(\C), \frac{1}{n}\tr)$, converges in $*$-distribution to a circular variable (see Theorem~\ref{thm:asconv}-(a)). When the entries form admissible SMSs with uniformly bounded moments of all orders, we are only able to establish an in-probability version (see Theorem~\ref{thm:asconv}-(b)). 
 
As a direct corollary to our results, we recover the aforementioned result of \cite{bose2014bulk} that the LSM of the $n^{-1/2}$-scaled Schur-Hadamard product of symmetric Toeplitz and Hankel matrices is the semi-circular law $\mu_{\semicirc}$ (see Corollary~\ref{cor:bm14}).

It is well-known that circular variables are $R$-diagonal, and they have $\mu_{\mathrm{circ}}$ as their Brown measure (see, e.g., Chapter 11 of \cite{Mingo2017}). In this context, let us recall that Voiculescu \cite{voiculescu1991limit} showed that matrices from the Ginibre ensemble (which are essentially IID matrices with Gaussian entries) converge in $*$-distribution to a circular variable, thus constructing the first matrix model for circular variables. The Schur-Hadamard product construction also gives a matrix model for circular variables, albeit with only $O(n)$ bits of randomness. 

\begin{figure}[!ht]
\centering
\begin{tabular}{ccc}
    \includegraphics[scale=0.38]{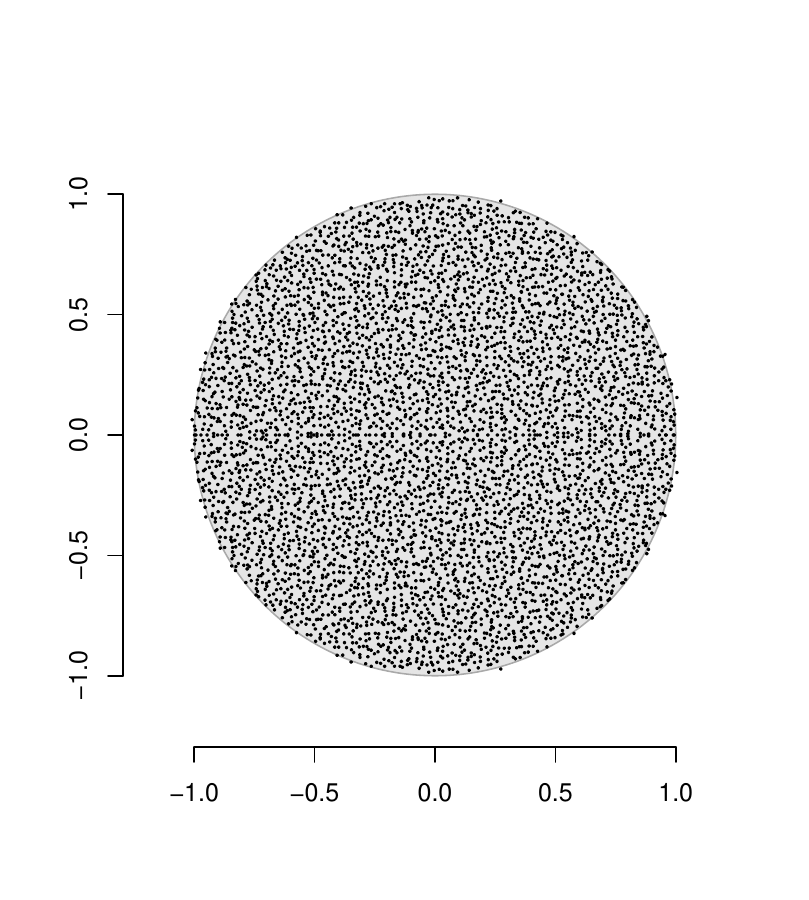} & \includegraphics[scale=0.38]{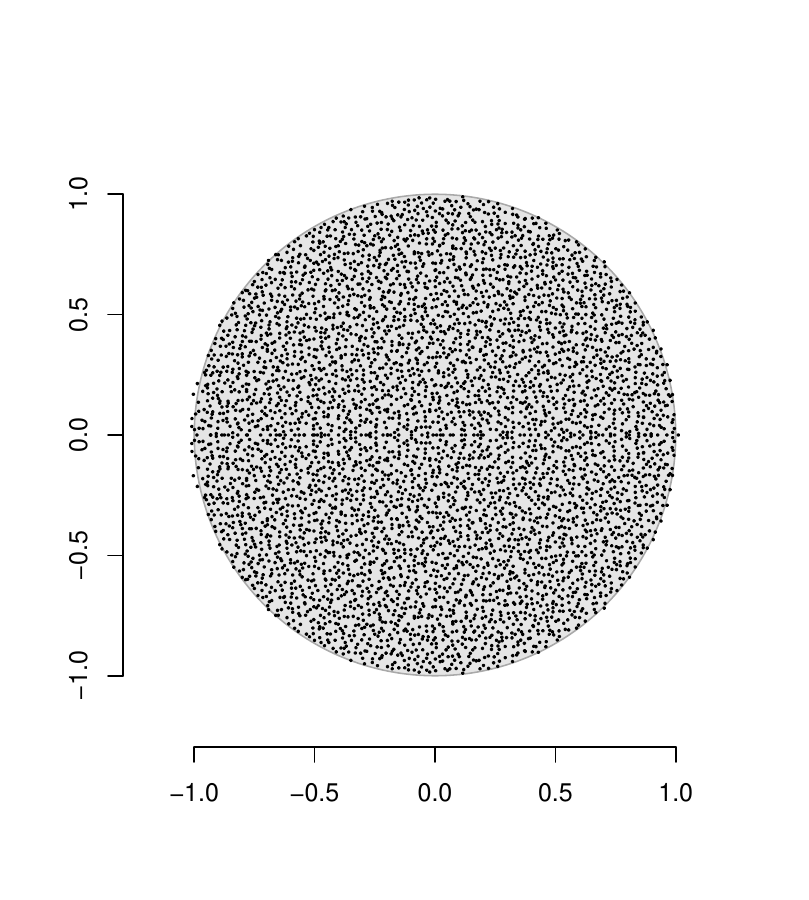} & \includegraphics[scale=0.38]{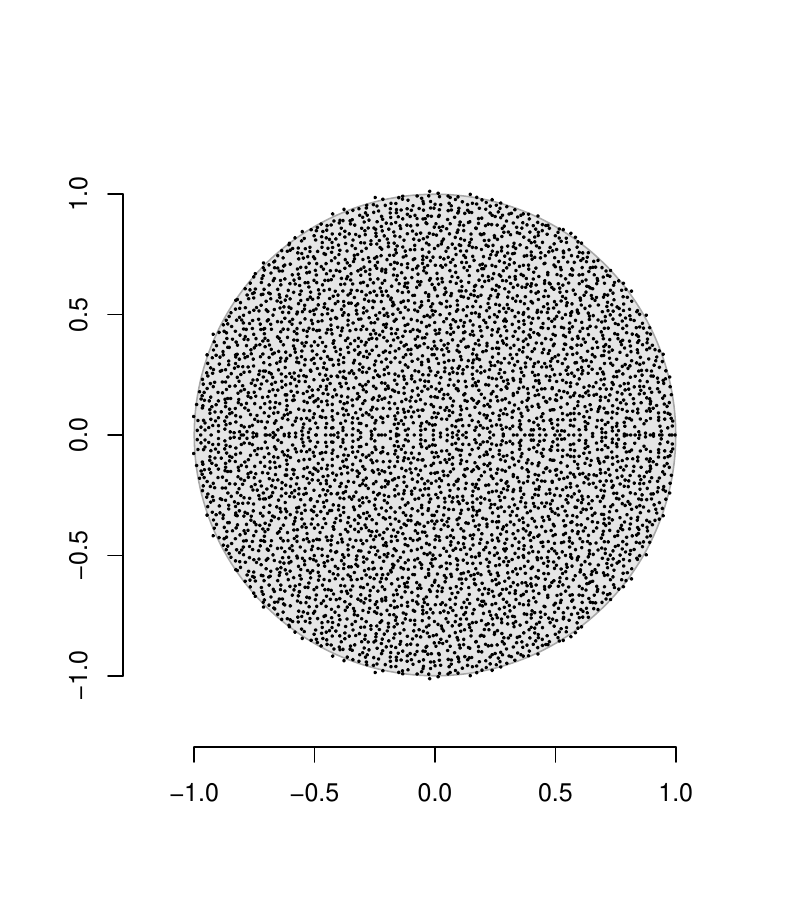}\\
    (a) & (b) & (c)
\end{tabular}
\caption{Spectrum of one realisation of $n^{-1/2}M_n$ for $n = 5000$ with $L_X(i, j) = i - j$, $L_Y(i, j) = i + j$. The component matrices are independent with each having (a) independent standard Gaussian entries, (b) independent Rademacher entries (i.e. random signs), and (c) entries from a copy of the SMS $\{\sqrt{2}\sin(2^n \pi U) : n \in \Z_+\}$, where $U \sim \mathrm{Uniform}(0, 1)$.}
\label{fig:toephank}
\end{figure}

Another main result of \cite{bose2014bulk} is that if $L_X$ and $L_Y$ determine the Wigner link-function $L_W(i, j) = (i\wedge j, i\vee j)$ in the sense that
\begin{equation}\label{eq:BM-cond}
    (L_X(i, j), L_Y(i, j)) = (L_X(i', j'), L_Y(i', j')) \implies L_W(i, j) = L_W(i', j') \text{ for all } i, j, i', j' \in \Z_+, 
\end{equation}
and some other regularity conditions hold on $L_X, L_Y$, then the LSM of $n^{-1/2}M_n$ is the semi-circular law $\mu_{\semicirc}$. Note that \eqref{eq:BM-cond} is equivalent to saying that $L_W = g \circ G$ for some map $g: \Z^d \times \Z^{d'} \rightarrow \Z_+^2$, where $G$ is the map $(i, j) \mapsto (L_X(i, j), L_Y(i, j))$. The symmetric Toeplitz and Hankel link-functions satisfy this property.

A similar result holds in the nonsymmetric case. We establish in Theorem~\ref{thm:gen} that if the map
\begin{equation*}
    (i, j) \mapsto (L_X(i, j), L_Y(i, j))
\end{equation*}
is injective, plus some regularity assumptions on $L_X$ and $L_Y$ hold, then $n^{-1/2}M_n$ converges in $*$-distribution to a circular variable. This motivates us to make the following generalisation of Conjecture~\ref{conj1}.

\begin{conj}\label{conj2}
    Suppose that Assumptions~\ref{assm:propB} and \ref{assm:genlink2} hold. Also, suppose that $\{x_{\alpha}\}_{\alpha \in \Z^d}$ and $\{ y_{\alpha}\}_{\alpha \in \Z^{d'}}$ are independent collections of i.i.d. random variables with zero mean and unit variance. Then the ESM of $n^{-1/2}M_n$ converges weakly almost surely to the circular law $\mu_{\mathrm{circ}}$.
\end{conj}

\begin{remark}
The condition that the map $(i, j) \mapsto (L_X(i, j), L_Y(i, j))$ is injective implies that the entries of $M_n$ are uncorrelated. Indeed, because of our assumptions on the input random variables, we have
\begin{align} \label{eq:cov} \nonumber
    \cov(M_{n, i, j}, M_{n, i', j'}) &= \E M_{n, i, j} M_{n, i', j'} \\ \nonumber
                                     &= \E x_{L_X(i, j)} x_{L_X(i', j')} \E y_{L_Y(i, j)} y_{L_Y(i', j')} \\
                                     &= \cov(x_{L_X(i, j)}, x_{L_X(i', j')}) \cov(y_{L_Y(i, j)}, y_{L_Y(i', j')}).
\end{align}
Now if $(i, j) \ne (i', j')$, then either $L_X(i, j) \ne L_X(i', j')$, or $L_Y(i, j) \ne L_Y(i', j')$. Thus one of the covariances in \eqref{eq:cov} must vanish.

However, just being uncorrelated is not enough as the simulations of Figure~\ref{fig:badXY} suggest. There we take $L_X(i, j) = i + j$ and $L_Y(i, j) = j$ which makes the map $(i, j) \mapsto (L_X(i, j), L_Y(i, j))$ injective. We will impose a regularity condition that would disallow link-functions like $L_Y(i, j) = j$ which make an input random variable appear too many times in a column or row (see Assumption~\ref{assm:propB}).

\begin{figure}[!ht]
\centering
\begin{tabular}{ccc}
    \includegraphics[scale=0.38]{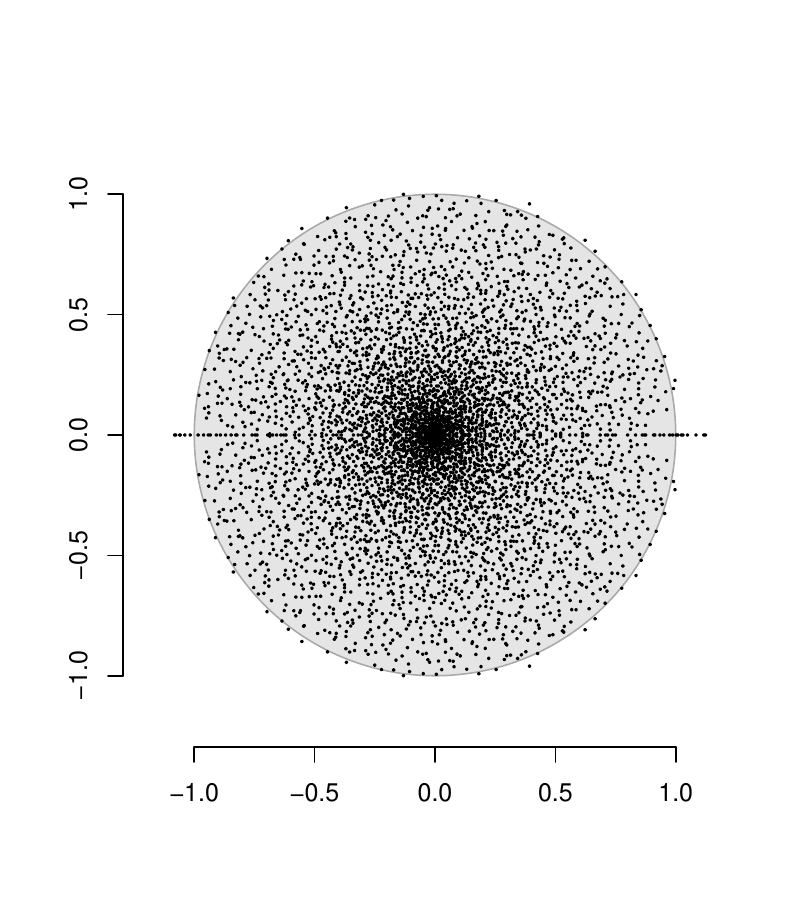} & \includegraphics[scale=0.38]{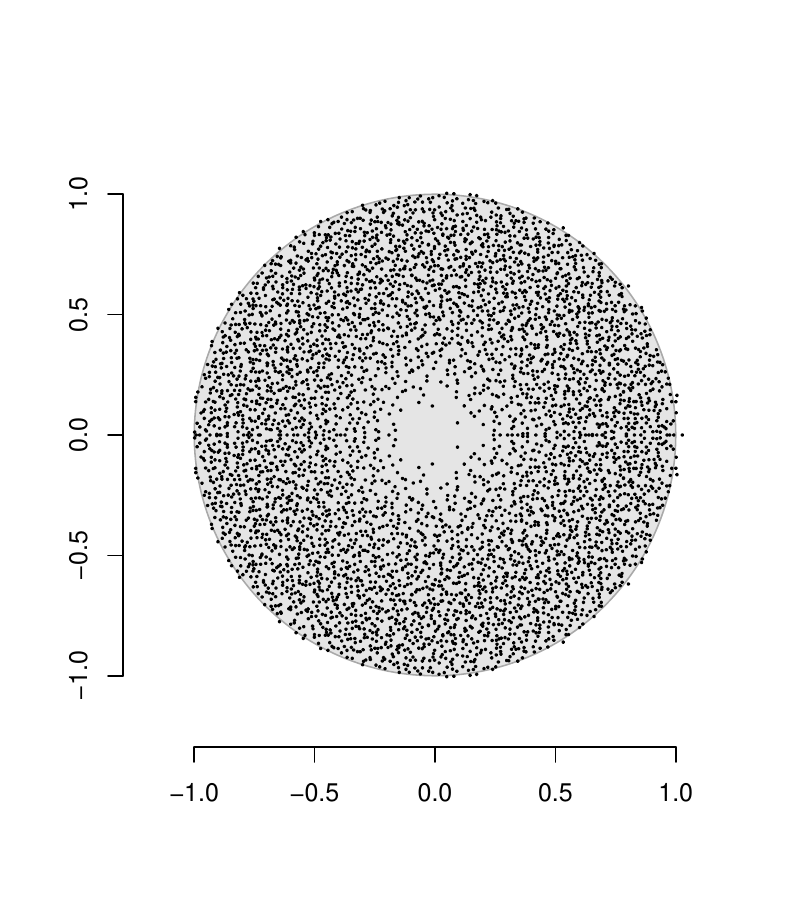} & \includegraphics[scale=0.38]{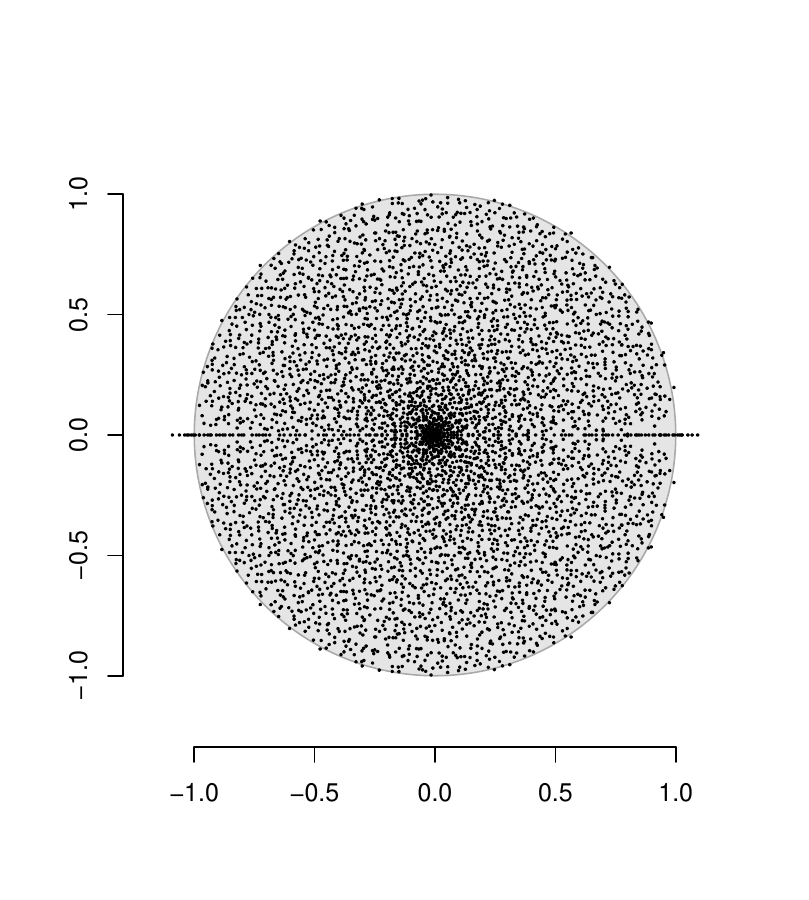}\\
    (a) & (b) & (c)
  \end{tabular}
  \caption{Spectrum of one realisation of $n^{-1/2}M_n$ for $n = 5000$ with $L_X(i, j) = i + j$, $L_Y(i, j) = j$. The component matrices are independent with each having (a) independent standard Gaussian entries, (b) independent Rademacher entries (i.e. random signs), and (c) entries from a copy of the SMS $\{\sqrt{2}\sin(2^n \pi U) : n \in \Z_+\}$, where $U \sim \mathrm{Uniform}(0, 1)$. These simulations suggest a lack of universality in this model, i.e. the LSM of $n^{-1/2}M_n$, if it exists, possibly depends on the distribution of the entries beyond dependence on just the first two moments.}
\label{fig:badXY}
\end{figure}
\end{remark}

\begin{remark}
In two special cases, Conjecture~\ref{conj2} can be proved easily:
\begin{enumerate}
    \item[(i)] $L_X(i, j) = L_Y(i, j) = (i, j)$.
    \item[(ii)] $L_X(i, j) = (i, j)$, the common distribution of the $x_{i, j}$'s is symmetric about $0$, and $(y_{\alpha})_{\alpha \in \Z^{d'}}$ is an i.i.d. Rademacher sequence. 
\end{enumerate}
In both cases, it is easy to see that $M_n$ is an IID matrix, and hence the LSM of $n^{-1/2}M_n$ is the circular law.
\end{remark}

The rest of the paper is organised as follows. In Section~\ref{sec:prelim}, we provide a review of the free-probabilistic concepts needed for our main results. In Section~\ref{sec:main}, we present our main results: In Sections~\ref{sec:main-gen} and \ref{sec:main-sms}, we state our assumptions on the input random variables and prove some general results. In Section~\ref{sec:toephank}, we state and prove the $*$-convergence result for the Schur-Hadamard product of nonsymmetric Toeplitz and Hankel matrices. Then, in Section~\ref{sec:gen}, we show how the proof technique for the Toeplitz and Hankel case extends to more general link-functions. Finally, in Section~\ref{sec:main-asconv}, we prove in-probability and almost sure versions of our $*$-convergence results.

\section{A quick review of some free-probabilistic concepts}\label{sec:prelim}
A non-commutative probability space is a pair $(\A, \state)$, where $\A$ is a unital algebra over $\C$, and $\state$ is a \emph{state}, i.e.  a linear functional on $\A$ such that $\state (\bone_{\A}) = 1$ (here $\bone_{\A}$ is the unit of $\A$). 

A $*$-probability space is a non-commutative probability space $(\A, \state)$, where $\A$ is a unital $*$-algebra, and the state $\state$ is \emph{positive}, i.e.
\begin{equation*}
    \state(a^*a) \ge 0 \text{ for all } a \in \A.
\end{equation*}
The state $\state$ plays the same role in non-commutative probability as the expectation operator does in classical probability.

A natural example comes from the unital $*$-algebra $\mathcal{M}_n(\C)$ of $n\times n$ matrices over $\C$ (the identity matrix serves as the unit while the $*$-operation is given by taking conjugate transpose). This becomes a $*$-probability space when endowed with the state $\frac{1}{n}\tr$.

To deal with random matrices, given a probability space $(\Omega, \mathcal{B}, \mathbb{P})$, one may consider the unital $*$-algebra $\mathcal{M}_n(L^{\infty, -}(\Omega, \mathbb{P}))$  of random matrices whose entries are in $L^{\infty, -}(\Omega, \mathbb{P}) = \cap_{1 \le p < \infty}L^p(\Omega, \mathbb{P})$, the space of random variables with all moments finite. Equipped with the state $\frac{1}{n}\E\tr$, where $\E$ denotes expectation with respect to $\mathbb{P}$, this becomes a $*$-probability space.

\begin{definition}\label{def:freeness} (Free independence)
    Let $(\A, \state)$ be a non-commutative probability space and let $\{\A_i\}_{i \in I}$ be a collection of unital subalgebras of $\A$, indexed by a fixed set $I$. The subalgebras $\{\A_i\}_{i\in I}$ are called freely independent, or just free, if 
\begin{equation*}
    \state(a_1 \cdots a_k) = 0
\end{equation*}
for every $k\ge 1$, where
\begin{itemize}
    \item[(a)] $a_j\in\A_{i_j}$ for some $i_j \in I$;  

    \item[(b)] $\state(a_j)=0$ for every $1\leq j\leq k$; and 

    \item[(c)] neighbouring elements are from different subalgebras, i.e. $i_1 \neq i_2, i_2 \ne i_3, \ldots, i_{k - 1}\ne i_k$. 
\end{itemize} 
Elements $\{a_i\}_{i \in I}$ from $\A$ are called freely independent, or just free, if the unital subalgebras generated by each of them are free.
\end{definition}

From now on, $(\A, \state)$ will denote a $*$-probability space unless stated otherwise. For $a \in \A$, $\epsilon_i \in \{1, *\}, 1 \le i \le k, k\ge 1$, the number $\state(a^{\epsilon_1}\cdots a^{\epsilon_k})$ is called a \emph{$*$-moment} of $a$. If $a$ is \emph{self-adjoint}, i.e. if $a = a^*$, then $*$-moments reduce to \emph{moments} $\state(a^k)$.

\begin{definition}
    A semi-circular variable in a $*$-probability space is a self-adjoint element whose moments are given by
    \begin{equation*}
        \state(s^k) = \begin{cases}
            C_{k/2} & \text{ if $k$ is even},\\
            0 & \text{ if $k$ is odd},
        \end{cases}
    \end{equation*}
    where $C_{n} = \frac{1}{n + 1}\binom{2n}{n}$ is the $n$-th Catalan number. It turns out that $(\state(s^k))_{k \ge 1}$ is the moment sequence of the semi-circular law $\mu_\semicirc$.
\end{definition}
\begin{definition}\label{def:circular}
    A circular variable $c$ in a $*$-probability space $(\A, \state)$ is an element of the form
\[
    c = \frac{s_1 + i s_2}{\sqrt{2}},
\]
where $s_1$ and $s_2$ are freely independent semi-circular elements.
\end{definition}

\begin{definition}\label{def:*-dist}
The $*$-distribution of an element $a \in \A$ is the linear functional $\mu_a:\C\langle X, X^* \rangle \rightarrow \C$ such that
\[
    \mu_a(Q) = \state(Q(a, a^*)).
\]
This is equivalent to specifying all the $*$-moments 
\begin{equation*}
\{\state(a^{\epsilon_1} \cdots a^{\epsilon_k}) \mid \epsilon_i \in \{1, *\}, k \ge 1\}.
\end{equation*}
\end{definition}

The $*$-moments (and hence the $*$-distribution) of a circular variable are easy to calculate due to the fact that circular variables are the simplest examples of the so-called $R$-diagonal variables (see, e.g., Chapter 15 of \cite{Nica2006}).

To describe these $*$-moments, we first need to talk about non-crossing partitions. Given a totally ordered finite set $S$, we denote by $\P(S)$ the set of all partitions of $S$. This is in fact a lattice with respect to the reverse refinement partial order: given $\pi, \sigma \in \P(S)$, $\pi \le \sigma$ if each block of $\pi$ is contained in some block of $\sigma$.

Given a partition $\pi$, let $V_1, \ldots, V_k$ be its blocks, ordered by their smallest elements. $k$ is called the size of $\pi$ and is denoted by $|\pi|$. The smallest element of each block will be called a \emph{primary} element, and the rest of the elements \emph{secondary}. The smallest element of $V_1$ is the first primary element, that of $V_2$ the second primary element, and so on.

To give an example, let $S = \{2, 4, 5, 9, 11, 14\}$. Consider the partition $\pi_0 = \{\{2, 4, 11\}, \{5, 14\}, \{9\}\}$. In this case, we have three blocks: $V_1 = \{2, 4, 11\}, V_2 = \{5, 14\}$, and $V_3 = \{9\}$. The elements $2, 5, 9$ are the primary elements in that order, the rest are secondary elements.

A partition $\pi \in \P(S)$ is \emph{crossing} if there are elements $u_1 < u_2 < u_3 < u_4$ in $S$ such that $u_1, u_3 \in V_i$ and $u_2, u_4 \in V_j$ for two different blocks $V_i, V_j$ of $\pi$. The partition $\pi_0$ in the above example is crossing.

Let $NC(S)$ denote the set of all non-crossing partitions of $S$. This is a sub-lattice of $\P(S)$. A partition is called a pair-partition if all its blocks are doubletons. The set of all pair-partitions of $S$ will be denoted by $\P_2(S)$. Similarly, the set of all non-crossing pair-partitions of $S$ will be denoted by $NC_2(S)$.

We will typically take $S = [k] = \{1, \ldots, k\}$ for some integer $k$. We will write $\P(k), \P_2(k), NC(k)$, and $NC_2(k)$ instead of $\P([k]), \P_2([k])$, etc. The Catalan numbers introduced earlier count non-crossing partitions: $\# NC(k) = \# NC_2(2k) = C_k$.

We are now ready to describe the $*$-moments of a circular variable.
\begin{prop}\label{prop:circ}
Let $c$ be a circular variable. Then
\begin{equation*}
    \state(c^{\epsilon_1} \cdots c^{\epsilon_k}) = \begin{cases}
        \sum_{\pi \in NC_2(k)} \prod_{\{r, s\} \in \pi} (1 - \delta_{\epsilon_r, \epsilon_s}) & \text{ if $k$ is even,} \\
        0 & \text{ if $k$ is odd}.
    \end{cases}
\end{equation*}
\end{prop}
Circular variables have $\mu_{\cir}$ as their Brown measure, which is a generalisation of spectral measures of self-adjoint operators to non-normal operators. We will not define Brown measures here and, instead, refer the interested reader to Chapter 11 of \cite{Mingo2017}.

Now we discuss the notion of convergence in $*$-distribution (also called convergence in $*$-moments). We will sometimes abbreviate this to just $*$-convergence.
\begin{definition}\label{def:*-conv}
    Let $(\A_n, \state_n)$ be a sequence of $*$-probability spaces. We say that $a_n \in \A_n$ converges in \emph{$*$-distribution} to $a \in (\A, \state)$ if for all $Q \in \C\langle X, X^* \rangle$, one has
\begin{equation*}
    \mu_{a_n}(Q) \xrightarrow{n \rightarrow \infty} \mu_a(Q).
\end{equation*}
This is equivalent to the requirement that for any $k \ge 1$ and symbols $\epsilon_1, \ldots, \epsilon_k \in \{1, *\}$, one has convergence of the corresponding $*$-moments:
\begin{equation*}
    \state_n(a_n^{\epsilon_{1}} \cdots a_n^{\epsilon_{k}}) \xrightarrow{n \rightarrow \infty} \state(a^{\epsilon_{1}} \cdots a^{\epsilon_{k}}).
\end{equation*}
\end{definition}

As discussed earlier, a natural host space for $n\times n$ random matrices is the $*$-probability space $(\A_n, \state_n) = (\mathcal{M}_n(L^{\infty, -}(\Omega, \mathbb{P})), \frac{1}{n}\E \tr)$. In Section~\ref{sec:main-gen}, we will be concerned with $*$-convergence of $n^{-1/2} M_n$ as a member of this $*$-probability space. In Section~\ref{sec:main-asconv}, we will consider almost sure $*$-convergence of $n^{-1/2}M_n$ as a member of the $*$-probability space $(\mathcal{M}_n(\C), \frac{1}{n}\tr)$.

\section{Main results}\label{sec:main}
We will first work out some general results when our matrices have entries that are independent random variables with zero mean, unit variance, and uniformly bounded moments of all orders. Then, in Section~\ref{sec:main-sms}, we will allow the entries to form SMSs. Our experience from Section~\ref{sec:main-gen} will lead us naturally to an admissibility condition on the SMSs, which will be used as a substitute for independence in killing certain joint moments of the entries.
\subsection{Some general results}\label{sec:main-gen}

\begin{assumption}\label{assm:entries}
    The input random variables $\{x_{\alpha}\}_{\alpha \in \Z^d}$ are independent with zero mean, unit variance, and uniformly bounded moments of all orders, i.e. for any $p \ge 1$,
    \begin{equation*}
    \sup_{\alpha \in \Z^d} \E |x_{\alpha}|^p < \infty.
    \end{equation*}
    Note that this is true if $\{x_\alpha\}_{\alpha \in \Z^d}$ is an i.i.d. collection with zero mean, unit variance, and all moments finite.
    We make the same assumptions on $\{y_\beta\}_{\beta \in \Z^{d'}}$. Moreover, we assume that $\{x_\alpha\}_{\alpha \in \Z^d}$ and $\{y_\beta\}_{\beta \in \Z^{d'}}$ are independent.
\end{assumption}
We note that these assumptions do imply that $M_n \in \mathcal{M}_n(L^{\infty, -}(\Omega, \mathbb{P}))$.

For $\epsilon \in \{1, *\}$, let 
\[
    L^\epsilon(i, j) = \begin{cases}
        L(i, j) & \text{ if } \epsilon = 1, \\
        L(j, i) & \text{ if } \epsilon = *.
    \end{cases}
\]
If $A = ((a_{L(i, j)}))$ is a patterned matrix, then we may write, using the above notation, $A^\epsilon_{ij} = a_{L^\epsilon(i, j)}$.

We now expand the $*$-moments of $n^{-1/2}M_n$. Below $\bi$ denotes a multi-index $(i_1, i_2, \ldots, i_k) \in [n]^k$ and $\epsilon_i \in \{1, *\}, 1 \le i \le k$, are fixed symbols.
\begin{align}\label{eq:expansion}\nonumber
    \state_n(&(n^{-1/2}M_n)^{\epsilon_1} \cdots (n^{-1/2}M_n)^{\epsilon_k}) \\ \nonumber
             &= \frac{1}{n^{1 + k/2}}\E \tr (M_n^{\epsilon_1} \cdots M_n^{\epsilon_{k}}) \\ \nonumber
                      &= \frac{1}{n^{1 + k/2}} \sum_{\bi \in [n]^k} \E M^{\epsilon_1}_{n, i_1i_2} \cdots M^{\epsilon_{k}}_{n, i_{k}i_1} \\ \nonumber
                      &= \frac{1}{n^{1 + k/2}} \sum_{\bi \in [n]^k} \E X^{\epsilon_1}_{n, i_1i_2} \cdots X^{\epsilon_{k}}_{n, i_{k}i_1} Y^{\epsilon_1}_{n, i_1i_2} \cdots Y^{\epsilon_{k}}_{n, i_{k}i_1} \\ \nonumber
                      &= \frac{1}{n^{1 + k/2}} \sum_{\bi \in [n]^k} \E x_{L_X^{\epsilon_1}(i_1, i_2)} \cdots x_{L_X^{\epsilon_{k}}(i_{k}, i_1)} \E y_{L_Y^{\epsilon_1}(i_1, i_2)} \cdots y_{L_Y^{\epsilon_{k}}(i_{k}, i_1)} \\
                      &= \frac{1}{n^{1 + k/2}} \sum_{\bi \in [n]^k} \E x^{\epsilon}_{\bi} \E y^{\epsilon}_{\bi},
\end{align}
where $x^{\epsilon}_{\bi}$ and $y^{\epsilon}_{\bi}$ are convenient shorthands for $x_{L_X^{\epsilon_1}(i_1, i_2)} \cdots x_{L_X^{\epsilon_{k}}(i_{k}, i_1)}$ and $y_{L_Y^{\epsilon_1}(i_1, i_2)} \cdots y_{L_Y^{\epsilon_{k}}(i_{k}, i_1)}$, respectively.

To any $\bi \in [n]^k$, associated is a partition of $[k]$ induced by the values of $L$ for fixed $\epsilon_1, \ldots, \epsilon_k$: $u, v$ belong to the same block if $L^{\epsilon_u}(i_u, i_{u + 1}) = L^{\epsilon_v}(i_v, i_{v + 1})$, where $i_{k + 1} \equiv i_1$. For a partition $\pi \in \P(k)$, we denote by $[\pi]_{L, \epsilon} \equiv [\pi]_L$ the set of all multi-indices $\bi \in [n]^k$, for which the associated partition is $\pi$. To give an example, let $k = 4$, $n = 6$, $\epsilon_1 = \epsilon_3 = 1$, $\epsilon_2 = \epsilon_4 = *$, and $L$ be the Toeplitz link-function. Then the multi-index $\bi = (5, 6, 1, 6)$ is associated with the partition $\pi = \{ \{1, 4\}, \{2, 3\} \}$ of $[4]$, since $L^{\epsilon_1}(i_1, i_2) = -1 = L^{\epsilon_4}(i_4, i_1)$ and $L^{\epsilon_2}(i_2, i_3) = -5 = L^{\epsilon_3}(i_3, i_4)$. We will sometimes refer to $L^{\epsilon_u}(i_u, i_{u + 1})$ as an \emph{$L$-value} at location $u$. Thus, in the above example, $-1$ and $-5$ are $L$-values in $\bi = (5, 6, 1, 6)$, both of which are repeated twice ($-1$ at locations $1$ and $4$, and $-5$ at locations $2$ and $3$). Finally, note that $\{[\pi]_L\}_{\pi \in \P(k)}$ is a partition of $[n]^k$.

Using these concepts, we rewrite \eqref{eq:expansion} as
\begin{align}\label{eq:expansionfinal}\nonumber
    \state_n((n^{-1/2}M_n)^{\epsilon_1} &\cdots (n^{-1/2}M_n)^{\epsilon_k}) \\
             &= \frac{1}{n^{1 + k/2}} \sum_{\pi, \pi' \in \P(k)} \sum_{\bi \in [\pi]_{X} \cap [\pi']_{Y}} \E x^{\epsilon}_{\bi} \E y^{\epsilon}_{\bi},
\end{align}
where we use the shorthand $[\pi]_{X} \equiv [\pi]_{L_X}$.

Let $\bi \in [\pi]_X$. $i_1$ and $i_{r + 1}$ for every primary element $r$ of $\pi$ are called \emph{generating indices} (here $i_{k + 1} \equiv i_1$). As there are $|\pi|$ many primary elements, the total number of generating indices in $\bi$ is $|\pi|$ if $k$ is a primary element of $\pi$ (i.e. $\{k\}$ is a block of $\pi$), and $|\pi| + 1$ otherwise.

The following assumption is crucial for analysing symmetric patterned random matrices (see, e.g., \cite{bose2018patterned}). We will also need it here.
\begin{assumption}\label{assm:propB}
    For a link-function $L$, let
    \begin{align}\label{eq:propB}\nonumber
        \Delta_{L} := \sup_n \sup_{t \in \mathrm{range}(L)} \max\bigg \{ \sup_{j \in [n]} \#\{i \mid i \in [n], L(i, j) = t\}, \sup_{i \in [n]} \#\{j \mid j \in [n], L(i, j) = t\} \bigg\}.
    \end{align}
    Thus $\Delta_L < \infty$ means that the total number of times a particular input variable can appear in a row or column is uniformly bounded.
    We assume that $\max\{\Delta_{L_X}, \Delta_{L_Y}\} < \infty$.
\end{assumption}
Note that for both the nonsymmetric Toeplitz and the Hankel link-functions, $\Delta_L = 1$. An important consequence of having $\Delta_L < \infty$ is that any non-generating index can be determined, up to a bounded number of choices, from its predecessor generating indices. We record this in the following lemma, which will be used repeatedly later.
\begin{lemma}\label{lem:non-gen}
Suppose $\Delta_L < \infty$. Let $\bi \in [\pi]_L$. Then one can determine any non-generating $i_s$, up to a bounded number of choices, from all previous generating $i_r, r < s$.
\end{lemma}
\begin{proof}
Let $i_{s}$ be a non-generating index. Then $s \ne 1$. Further, $1$ is a primary element, so that $i_2$ is generating. Thus we must have $s > 2$. Also, $s - 1$ must be a secondary element. Thus there exists some $s' < s - 1$ such that $s', s - 1$ belong to the same block of $\pi$. Now we have the relation
\begin{equation}\label{eq:gen-back-repr}
    L^{\epsilon_{s'}}(i_{s'}, i_{s' + 1}) = L^{\epsilon_{s - 1}}(i_{s - 1}, i_s).
\end{equation}
As $\Delta_L < \infty$, $i_s$ can be determined from \eqref{eq:gen-back-repr}, up to a bounded number of choices, for any fixed $i_{s'}, i_{s' + 1}, i_{s - 1}$. 

We now use an inductive argument. Taking $i_{s_1}$ to be the first non-generating index, we see that $i_{s_1}$ can be determined from the choices of the preceding indices all of which are generating. Assume that the first $t$-many non-generating indices can be determined from their predecessor generating indices. By the argument given in the previous paragraph, the $(t + 1)$-th non-generating index $i_{s_{t + 1}}$ can be determined by previously occurring indices. But since all previously occurring non-generating indices can be determined by their predecessor generating indices, we conclude that $i_{s_{t + 1}}$ is determined by its predecessor generating indices.
\end{proof}
The proof of the following lemma is standard. We include it here for completeness.
\begin{lemma}\label{lem:part}
    Let $k \ge 1$ and $\pi \in \mathcal{P}(k)$. If $\Delta_L < \infty$, then $\# [\pi]_{L} \le C_{L, k}\, n^{1 + |\pi|}$ for some constant $C_{L, k} > 0$, depending only on the link-function $L$ and $k$.
\end{lemma}
\begin{proof}
    $\bi \in [\pi]_L$ implies that if $u, v$ belong to the same block of $\pi$, then we have the constraint
    \begin{equation*}
        L^{\epsilon_u}(i_u, i_{u + 1}) = L^{\epsilon_v}(i_v, i_{v + 1}).
    \end{equation*}
    Therefore, if $\Delta_L < \infty$, then once the generating indices of $\bi$ have been chosen in one of at most $n^{|\pi| + 1}$ many ways, the non-generating indices can be chosen in at most $\Delta_L^{k - |\pi|}$ ways. We can thus take $C_{L, k} = \Delta_L^{k - 1}$.
\end{proof}
\begin{lemma}\label{lem:nonpair}
    Suppose that Assumptions~\ref{assm:entries} and \ref{assm:propB} hold. If $\pi$ and $\pi'$ are not both pair-partitions, then
\begin{equation}\label{eq:nonpair}
    \frac{1}{n^{1 + k/2}} \sum_{\bi \in [\pi]_{X} \cap [\pi']_{Y}} \E x^{\epsilon}_{\bi} \E y^{\epsilon}_{\bi} = o(1).
\end{equation}
\end{lemma}
\begin{proof}
    If $\pi$ or $\pi'$ have any singleton blocks, then, by the centredness and independence of the input variables, $\E x_{\bi}^{\epsilon} \E y_{\bi}^{\epsilon} = 0$ for any $\bi \in [\pi]_X \cap [\pi']_Y$, and hence the left-hand side of \eqref{eq:nonpair} is trivially zero. 

So we assume that all the blocks of $\pi$ and $\pi'$ are of size at least $2$. Since not both are pair-partitions, one of them, say $\pi$, has a block of size at least $3$. 
Since the input random variables have uniformly bounded moments of all orders, $|\E x_{\bi}^{\epsilon} \E y_{\bi}^{\epsilon}| = O(1)$ by H\"{o}lder's inequality. Thus it suffices to show that
\[
    \# [\pi]_X \cap [\pi]_Y = o(n^{1 + k/2}).
\]
This follows from Lemma~\ref{lem:part} because
\[
    \# [\pi]_X = O(n^{1 + |\pi|}) = o(n^{1 + k/2}), 
\]
where we have used the bound $|\pi| < k/2$ which is true since each block of $\pi$ has size at least $2$ and there is at least one block of size at least $3$.
\end{proof}
From Lemma~\ref{lem:nonpair}, we immediately get that for odd $k$,
\begin{equation}\label{eq:oddmoments}
    \state_n((n^{-1/2}M_n)^{\epsilon_1} \cdots (n^{-1/2}M_n)^{\epsilon_k}) = o(1). 
\end{equation}
From now on, we shall work with $k$ even and write $2k$ instead of $k$ in the forthcoming expressions. 

Note that for pair-partitions $\pi, \pi' \in \P_2(2k)$, if $\bi \in [\pi]_X \cap [\pi']_Y$, then $\E x_{\bi}^\epsilon = \E y_{\bi}^{\epsilon} = 1$. It follows that
\begin{equation}\label{eq:pp}
    \frac{1}{n^{1 + k}} \sum_{\bi \in [\pi]_X \cap [\pi']_Y} \E x_{\bi}^{\epsilon} \E y_{\bi}^\epsilon = \frac{\# [\pi]_X \cap [\pi']_Y}{n^{1 + k}}.
\end{equation}
Motivated by the respective definitions in \cite{bose2014bulk}, we make the following definitions.
\begin{definition}
    Two link-functions $L_X$ and $L_Y$ are said to be \emph{compatible} if for any $\pi \ne \pi' \in \P_2(2k)$, we have
    \begin{equation*}
        \frac{\# [\pi]_X \cap [\pi']_Y}{n^{1 + k}} = o(1).       
    \end{equation*}
\end{definition}

\begin{definition}
    Two link-functions $L_X$ and $L_Y$ are said to be \emph{jointly circular} if
    \begin{equation*}
        \frac{\# [\pi]_X \cap [\pi]_Y}{n^{1 + k}} \xrightarrow{n\rightarrow \infty} \begin{cases}
            \prod_{\{r, s\} \in \pi} (1 - \delta_{\epsilon_r, \epsilon_s}) & \text{ if } \pi \in NC_2(2k), \\
            0 & \text{ if } \pi \in \P_2(2k)\setminus NC_2(2k).
        \end{cases}
    \end{equation*}
\end{definition}
\begin{theorem}\label{thm:main}
    Suppose that Assumptions~\ref{assm:entries} and \ref{assm:propB} hold.
    If the two link-functions $L_X$ and $L_Y$ are compatible and jointly circular, then $n^{-1/2}M_n$ converges in $*$-distribution to a circular variable in expected normalised trace.
\end{theorem}

\begin{proof}
    From \eqref{eq:oddmoments}, we know that if $k$ is odd, then 
    \begin{equation*}
        \state_n((n^{-1/2}M_n)^{\epsilon_1} \cdots (n^{-1/2}M_n)^{\epsilon_k}) \xrightarrow{n\rightarrow \infty} 0.
    \end{equation*}
    On the other hand, by Lemma~\ref{lem:nonpair} and \eqref{eq:pp}, we have
\begin{align*}
    \state_n((n^{-1/2}&M_n)^{\epsilon_1} \cdots (n^{-1/2}M_n)^{\epsilon_{2k}}) \\
                                        &= \sum_{\pi, \pi' \in \P_2(2k)}\frac{\# [\pi]_X \cap [\pi']_Y}{n^{1 + k}} + o(1) \\
                                        &= \sum_{\pi \in \P_2(2k)}\frac{\# [\pi]_X \cap [\pi]_Y}{n^{1 + k}} + o(1) \quad \text{(by compatibility)} \\
                                        &=\sum_{\pi \in NC_2(2k)}\prod_{\{r, s\} \in \pi} (1 - \delta_{\epsilon_r, \epsilon_s}) + o(1) \quad \text{(by joint circularity)}.
\end{align*}
Thus the $*$-moments of $n^{-1/2}M_n$ converge to those of a circular variable as described in Proposition~\ref{prop:circ}. This completes the proof.
\end{proof}

In Sections~\ref{sec:toephank} and \ref{sec:gen}, we will find conditions on $L_X$ and $L_Y$ that make them compatible and jointly circular.

\subsection{Entries that form SMSs}\label{sec:main-sms}

\begin{definition}\label{defn:sms}
    A collection of random variables $\{x_{\alpha}\}_{\alpha \in I}$ is called an SMS if for all distinct $\alpha_k \in I$ and $\eta_k \in \{1, 2\}$, $1 \le k \le n$, $n \ge 1$, one has
    \[
        \E \prod_{k = 1}^n x_{\alpha_k}^{\eta_k} = \begin{cases}
            1 & \text{ if } \eta_1 = \cdots = \eta_k = 2, \\
            0 & \text{ otherwise.}
        \end{cases}
    \]
\end{definition}

A trivial example of an SMS would be a collection of independent random variables with zero mean and unit variance. More generally, any martingale difference sequence $(D_n)_{n \in Z_+}$ is an SMS, provided that $\E D_1^2 = 1$, and $\E [D_n^2 \mid D_{n - 1}, \ldots, D_1] = 1$ for all $n \ge 2$\footnote{Examples of such systems are simple to construct. For instance, set $D_1 = 2 \sqrt{2\alpha_1 + 1}(Y_1 - \frac{1}{2})$ with $Y_1 \sim \mathrm{Beta}(\alpha_1, \alpha_1)$, $\alpha_1 > 0$, and for $n \ge 2$, take
\[
    D_n \mid D_{n - 1}, \ldots, D_1 = 2\sqrt{2\alpha_n + 1} \bigg(Y_n - \frac{1}{2}\bigg),
\]
where $Y_n \sim \mathrm{Beta}(\alpha_n, \alpha_n)$ and $\alpha_n = \alpha_n(D_{n - 1}, \ldots, D_1)$ is some positive measurable function of $D_1, D_2, \ldots, D_{n - 1}$.}. Another nontrivial example, which is not a martingale difference sequence, is the collection $\{\sqrt{2}\sin(2^n \pi U) : n \in \Z_+\}$, where $U \sim \mathrm{Uniform}(0, 1)$. 

If $\{x_{\alpha}\}_{\alpha \in \Z^d}$ is an SMS, $L:\Z_+^2 \rightarrow \Z^d$ is a link-function, and $\bi \in [\pi]_L$, where $\pi$ is a pair-partition, then it is clear that $\E x_{\bi}^\epsilon = 1$. To deal with other partitions, we need further assumptions.

\begin{definition}\label{defn:asms}
    Say that an SMS $\{x_{\alpha}\}_{\alpha \in I}$, with $x_{\alpha}$ having all moments finite for any $\alpha \in I$, is \emph{admissible} if for all distinct $\alpha_k \in I$ and $\eta_k \ge 1$, $1 \le k \le n$, $n \ge 1$, one has
\[
    \E \prod_{k = 1}^n x_{\alpha_k}^{\eta_k} = 0
\]
if there is at least one $k$ such that $\eta_k = 1$.
\end{definition}

A trivial example of an admissible SMS would be a collection of independent random variables with zero mean, unit variance, and all moments finite. Of course, SMSs arising from more general martingale difference sequences do not necessarily satisfy the admissibility condition. Finally, one can verify that the SMS $\{\sqrt{2}\sin(2^n \pi U) : n \in \Z_+\}$, where $U \sim \mathrm{Uniform}(0, 1)$, is also admissible. 

If $\{x_{\alpha}\}_{\alpha \in \Z^d}$ is an admissible SMS, $L:\Z_+^2 \rightarrow \Z^d$ is a link-function, and $\bi \in [\pi]_L$, where $\pi$ is a partition containing a singleton block, then it is clear that $\E x_{\bi}^\epsilon = 0$.

We tackled partitions that contain no singleton blocks and at least one block of size $3$ using Lemma~\ref{lem:part} and the assumption of uniformly bounded moments of all orders. Note that the collection $\{\sqrt{2}\sin(2^n \pi U) : n \in \Z_+\}$, where $U \sim \mathrm{Uniform}(0, 1)$, satisfies this by virtue of being uniformly bounded.

The discussions above show that Assumption~\ref{assm:entries} in Lemma~\ref{lem:nonpair} and Theorem~\ref{thm:main} can be replaced by the following weaker assumption.
\begin{assumption-alt}{assm:entries}\label{assm:entries-sms}
We assume that $\{x_{\alpha}\}_{\alpha \in \Z^d}$ is an admissible SMS with uniformly bounded moments of all orders, i.e. for any $p \ge 1$, one has
\[
    \sup_{\alpha \in \Z^d} \E |x_{\alpha}|^p < \infty.
\]
We make the same assumptions on $\{y_\beta\}_{\beta \in \Z^{d'}}$. Moreover, we assume that $\{x_\alpha\}_{\alpha \in \Z^d}$ and $\{y_\beta\}_{\beta \in \Z^{d'}}$ are independent.
\end{assumption-alt}

\subsection{Toeplitz and Hankel}\label{sec:toephank}
In what follows, we will often find constraints between generating indices. Such constraints will be referred to as \emph{killing} constraints, because existence of such a constraint will make $\# [\pi]_X \cap [\pi]_Y = o(n^{1 + k})$, thus killing its asymptotic contribution to the relevant $*$-moment.

Also, in non-crossing pair-partitions, there will always exist a block of adjacent elements. Such blocks will be referred to as \emph{good}. Once we remove a good block, there will be a new good block in the remaining partition (of the reduced set). For example, in the case $k = 3$, consider the non-crossing pair-partition $\{\{1, 6\}, \{2, 5\}, \{3, 4\}\}$. Here $\{3, 4\}$ is a good block. After we remove $\{3, 4\}$, we get a new good block $\{2, 5\}$ of adjacent elements in the reduced set $\{1, 2, 5, 6\}$.

We first prove a general result regarding joint-circularity. 
\begin{lemma}\label{lem:suff}
    Suppose the map $(i, j) \mapsto (L_X(i, j), L_Y(i, j))$ is injective and Assumption~\ref{assm:propB} holds. Then $L_X$ and $L_Y$ are jointly circular.
\end{lemma}
\begin{proof}
    Suppose that $\pi \in NC_2(2k)$. Then $\bi \in [\pi]_X \cap [\pi]_Y$ implies that for all $\{r, s\} \in \pi$, we have
    \begin{align*}
        L_X^{\epsilon_r}(i_r, i_{r + 1}) &= L_X^{\epsilon_s}(i_s, i_{s + 1}), \\
        L_Y^{\epsilon_r}(i_r, i_{r + 1}) &= L_Y^{\epsilon_s}(i_s, i_{s + 1}).
    \end{align*}
    We have two possibilities:
    \begin{enumerate}
        \item[(a)] $\epsilon_r = \epsilon_s$. In this case, we conclude by injectivity that $i_r = i_s$ and $i_{r + 1} = i_{s + 1}$.
        \item[(b)] $\epsilon_r \ne \epsilon_s$. In this case, we conclude that $i_r = i_{s + 1}$ and $i_{r + 1} = i_s$.
    \end{enumerate}
    Since $\pi$ is a non-crossing pair-partition, it has a good block of the form $\{u, u + 1\}$. In the case $\epsilon_u = \epsilon_{u + 1}$, we get the constraint $i_u = i_{u + 1} = i_{u + 2}$. In the case $\epsilon_u \ne \epsilon_{u + 1}$, we get the constraint $i_u = i_{u + 2}$. In the former case, we cannot choose the generating index $i_{u + 1}$ freely, i.e. we have a killing constraint; this means that the asymptotic contribution from $\pi$ is zero. On the other hand, in the second case, $i_{u + 1}$ is a free variable. If we remove the block $\{u, u + 1\}$ from $\pi$, then in the remaining partition of $[2k]\setminus\{u, u + 1\}$ there is another good block of adjacent elements. Apply the above argument again to that block. It is clear that we can have a potentially nonzero contribution only if $\epsilon_r \ne \epsilon_s$ for all $\{r, s\} \in \pi$. In the latter case, we get exactly $k + 1$ free variables, which shows that $\# [\pi]_X \cap [\pi]_Y \sim n^{k + 1}$. This means that
    \begin{equation*}
        \frac{\# [\pi]_X \cap [\pi]_Y}{n^{1 + k}} \xrightarrow{n \rightarrow \infty} \prod_{\{r, s\} \in \pi} (1 - \delta_{\epsilon_r, \epsilon_s})
    \end{equation*}
    if $\pi$ is a non-crossing pair-partition.

    Now suppose that $\pi$ is a crossing pair-partition. Pick $\{r, s\} \in \pi$ with the property that $s$ is the smallest secondary element. Then, as before, we have the contingencies (a) and (b). By the definition of $s$, we must have that $s - 1 \ge 1$ is primary and so $i_s$ is generating. Similarly, $i_r$ is also generating, because either $r = 1$ or $r - 1$ is primary (as $r < s$). Thus, in the case $\epsilon_r = \epsilon_s$, the constraint $i_r = i_s$ is a killing constraint. So we assume on the contrary. Then, since both $i_{r+1}$ and $i_s$ are generating, we have a killing constraint unless $s = r + 1$. If the latter is the case, we remove the good block $\{r, r + 1\}$ from $\pi$ and apply the above argument again. As $\pi$ is non-crossing, sooner or later we will end up with a partition having no good blocks and hence a nontrivial killing constraint. This means that
    \begin{equation*}
       \frac{\# [\pi]_X \cap [\pi]_Y}{n^{1 + k}} \xrightarrow{n \rightarrow \infty} 0 
   \end{equation*}
if $\pi$ is a crossing pair-partition.
\end{proof}

\begin{cor}\label{cor:toephank}
    The nonsymmetric Toeplitz and the Hankel link-functions are jointly circular.
\end{cor}
\begin{proof}
    Clearly, the map $(i, j) \mapsto (i - j, i + j)$ is injective.
\end{proof}
\begin{lemma}\label{lem:compatibility}
The nonsymmetric Toeplitz and the Hankel link-functions are compatible.
\end{lemma}
\begin{proof}
   Let us take $L_X(i, j) = i - j$ and $L_Y(i, j) = i + j$. Suppose that $\pi \ne \pi'$ are pair-partitions. Then the link-functions give us $2k$ constraints in total. We will show that under these $2k$ constraints, there can be at most $k$ elements that can be chosen freely.

    Consider the primary elements in $\pi$ and $\pi'$. Suppose these are \emph{not} the same for the two partitions. Consider the first element $r$ that is primary in $\pi$ but secondary in $\pi'$. Then $i_{r + 1}$ is generating according to $L_X$ but non-generating according to $L_Y$. Hence, choosing the previous generating indices according to $L_Y$ fixes $i_{r + 1}$. Since the previous generating indices agree for $\pi$ and $\pi'$, $i_{r + 1}$ can be determined up to boundedly many choices in terms of previous generating indices. This gives a nontrivial killing constraint.

    Now suppose that the primary elements are the same in both $\pi$ and $\pi'$. Then, since $\pi \ne \pi'$, there is an element $s$ that is secondary in both $\pi$ and $\pi'$ such that there are primary elements $r \ne r'$ with $\{r, s\} \in \pi$ and $\{r', s\} \in \pi'$. Let $s$ be the smallest element of this type. Without loss of generality, let us assume that $r' < r$. Then we have
    \begin{align*}
        (-1)^{1 - \delta_{\epsilon_r, \epsilon_s}} (i_r - i_{r + 1}) &= i_s - i_{s + 1}, \\
        i_{r'} + i_{r' + 1} &= i_s + i_{s + 1}.
    \end{align*}
    Adding the above two equations, we get
    \begin{equation}\label{eq:is-elim}
        2i_s = i_{r'} + i_{r' + 1} + (-1)^{1 - \delta_{\epsilon_r, \epsilon_s}} (i_r - i_{r + 1}).
    \end{equation}
    \textbf{Case I ($r = s - 1$):} Using \eqref{eq:is-elim}, we get that
    \begin{equation*}
        2i_{r + 1} = i_{r'} + i_{r' + 1} + (-1)^{1 - \delta_{\epsilon_r, \epsilon_s}} (i_r - i_{r + 1}),
    \end{equation*}
    which is a killing constraint. 
    \vskip5pt
    \noindent
    \textbf{Case II ($r < s - 1$):} If $s - 1$ is a primary element, then $i_s$ is generating, and we have from \eqref{eq:is-elim} a killing constraint. So we assume that $s - 1$ is secondary. Therefore there exists a primary element $r_1 < s - 1$ such that $\{r_1, s - 1\}$ is a block of both $\pi$ and $\pi'$. This gives, via injectivity of $(i, j) \mapsto (i - j, i + j)$,
    \[
        (i_{r_1}, i_{r_1 + 1}) = \begin{cases}
        (i_{s - 1}, i_{s}) & \text{if } \epsilon_{r_1} = \epsilon_{s - 1}, \\
        (i_{s}, i_{s - 1}) & \text{if } \epsilon_{r_1} \ne \epsilon_{s - 1}.
    \end{cases}
    \]
    In the former case, $i_{r_1 + 1} = i_s$, which, via \eqref{eq:is-elim}, leads to the relation
    \[
        2i_{r_1 + 1} = i_{r'} + i_{r' + 1} + (-1)^{1 - \delta_{\epsilon_r, \epsilon_s}} (i_r - i_{r + 1}).
    \]
    This is a killing constraint.
    In the latter case, $i_{r_1} = i_s$, which, via \eqref{eq:is-elim}, leads to
    \[
        2i_{r_1} = i_{r'} + i_{r' + 1} + (-1)^{1 - \delta_{\epsilon_r, \epsilon_s}} (i_r - i_{r + 1}).
    \]
    If $r_1 \le r + 1$, this gives a killing constraint for $i_{r + 1}$. If $r_1 > r + 1$ and $r_1 - 1$ is primary, then we again have a killing constraint. Finally, if $r_1 - 1$ is secondary, we apply the same argument as above with $r_1 - 1$ replacing the role of $s - 1$ to get a primary element $r_2$ such that $\{r_2, r_1 - 1\}$ is a block of both $\pi$ and $\pi'$, and so on. It is clear that by continuing this procedure we will eventually end up with some primary element $r_m, m \ge 2$, such that $\{r_m, r_{m - 1} - 1\}$ is a block of both $\pi$ and $\pi'$ and $r_m \le r + 1$. Then, no matter whether $\epsilon_{r_m}$ and $\epsilon_{r_{m - 1} - 1}$ are equal or not, we have some nontrivial killing constraint.
\end{proof}

Together, Corollary~\ref{cor:toephank} and Lemma~\ref{lem:compatibility} give us
our main result on nonsymmetric Toeplitz and Hankel matrices.
\begin{theorem}\label{thm:toephank}
    Under Assumption~\ref{assm:entries-sms}, if $L_X(i, j) = i - j$ and $L_Y(i, j) = i + j$, then $n^{-1/2}M_n$ converges in $*$-distribution to a circular variable in expected normalised trace.
\end{theorem}

\subsection{General link-functions}\label{sec:gen}
In the proof of Lemma~\ref{lem:compatibility}, we have used three facts about the nonsymmetric Toeplitz and the Hankel link-functions:
\begin{enumerate}
    \item[(i)] For both link-functions, one can determine any non-generating $i_s$, up to a bounded number of choices, from all previous generating $i_r, r < s$. This follows from the fact that these two link-functions satisfy Assumption~\ref{assm:propB} (via Lemma~\ref{lem:non-gen}).
    \item[(ii)] The map $(i, j) \mapsto (i - j, i + j)$ is injective.
    \item[(iii)] Equation \eqref{eq:is-elim} gives rise to killing constraints.
\end{enumerate}

In order to prove compatibility of two general link-functions $L_X$ and $L_Y$ along the lines of Lemma~\ref{lem:compatibility}, we need appropriate assumptions on them so as to have analogues of items (i) - (iii). Assumption~\ref{assm:propB} will give us (i). For (ii), we will assume the injectivity of the map $(i, j) \mapsto (L_X(i, j), L_Y(i, j))$, which we anyway need for joint circularity. Let us now look at the analogue of \eqref{eq:is-elim} in the general case.

As in the proof of Lemma~\ref{lem:compatibility}, consider the situation where $\pi$ and $\pi'$ have the same primary elements, and the elements $r, r', s$ encountered there, with $r \ne r'$ and $r, r' < s$. We then have the simultaneous equations
\begin{equation}\label{eq:gen-simultaneous-eq}
    L_X^{\epsilon_r}(i_r, i_{r + 1}) = L_X^{\epsilon_s}(i_s, i_{s + 1}), \quad L_Y^{\epsilon_{r'}}(i_{r'}, i_{r' + 1}) = L_Y^{\epsilon_s}(i_s, i_{s + 1}).
\end{equation}
Denote that map $(i, j) \mapsto (L_X(i, j), L_Y(i, j))$ by $G$. Let $\Ginv : \mathrm{range}(G) \rightarrow \Z_+^2$ denote the inverse of $G$. Write $\Ginv = (\Ginv_1, \Ginv_*)$. Then from \eqref{eq:gen-simultaneous-eq}, we get that
\begin{equation}\label{eq:gen0}
    i_s = \Ginv_{\epsilon_s}(L_X^{\epsilon_r}(i_r, i_{r + 1}), L_Y^{\epsilon_{r'}}(i_{r'}, i_{r' + 1})).
\end{equation}
which is the general version of \eqref{eq:is-elim}. As in the proof of Lemma~\ref{lem:compatibility}, this gives rise to various types of constraints for different combinations of $r, r'$ and $s$. We need to make sure that these are killing constraints. 
\begin{lemma}\label{lem:g}
    Suppose that Assumption~\ref{assm:propB} holds. Let $\epsilon \in \{1, *\}$. Then the equation
    \begin{equation}\label{eq:g}
        \Ginv_{\epsilon}(u, v) = w
    \end{equation}
    has an at most bounded number of solutions in $u$ (resp. $v$) when $v, w$ (resp. $u, w$) are fixed.
\end{lemma}
\begin{proof}
    Suppose that $\epsilon = 1$ and $v, w$ are given (the other cases can be tackled similarly). Then for any solution $u$ of \eqref{eq:g}, note that $w' = \Ginv_*(u, v)$ satisfies
    \[
        L_X(w, w') = u \text{ and } L_Y(w, w') = v.
    \]
    As $v, w$ are given, by Assumption~\ref{assm:propB}, there are an at most bounded number of solutions in $w'$ of the equation $L_Y(w, w') = v$. This gives the desired result since $u$ must satisfy $u = L_X(w, w')$.
\end{proof}
\begin{lemma}\label{lem:g-gen}
Let $\epsilon, \epsilon', \epsilon'' \in \{1, *\}$. Consider the equation
\begin{equation}\label{eq:g-gen}
    m = \Ginv_\epsilon(L_X^{\epsilon'}(i, j), L_Y^{\epsilon''}(k, l)).
\end{equation}
Then fixing any four variables among $i, j, k, l, m$ constrains the remaining variable to take an at most bounded number of values.
\end{lemma}
\begin{proof}
    If $i, j, k, l$ are fixed, then obviously $m$ is determined. If $j, k, l, m$ are fixed, then, by Lemma~\ref{lem:g}, $L_X^{\epsilon'}(i, j)$ can take an at most bounded number of values. But this means, by Assumption~\ref{assm:propB}, that $i$ can take an at most bounded number of values. The other cases are similar.
\end{proof}
Lemma~\ref{lem:g-gen} will help us get killing constraints from \eqref{eq:gen0} for most values of $r, r'$ and $s$. However, in certain degenerate cases we need to make extra assumptions.
\begin{assumption}\label{assm:genlink2}
    Assume that the link-functions $L_X$ and $L_Y$ are such that the map $G$ taking $(i, j) \mapsto (L_X(i, j), L_Y(i, j))$ is injective. Let $\Ginv : \mathrm{range}(G) \rightarrow \Z_+^2$ denote the inverse of $G$. Write $\Ginv = (\Ginv_1, \Ginv_*)$. Let $\epsilon, \epsilon', \epsilon'' \in \{1, *\}$. Consider the following equations obtained from \eqref{eq:g-gen} by equating some of the variables.
\begin{equation}\label{eq:g-spl-1}
    j = \Ginv_{\epsilon}(L_X^{\epsilon'}(i, j), L_Y^{\epsilon''}(k, l)).
\end{equation}
\begin{equation}\label{eq:g-spl-2}
    j = \Ginv_{\epsilon}(L_X^{\epsilon'}(l, j), L_Y^{\epsilon''}(k, l)).
\end{equation}
\begin{equation}\label{eq:g-spl-3}
    m = \Ginv_{\epsilon}(L_X^{\epsilon'}(l, j), L_Y^{\epsilon''}(k, l)).
\end{equation}
\begin{equation}\label{eq:g-spl-4}
    m = \Ginv_{\epsilon}(L_X^{\epsilon'}(l, j), L_Y^{\epsilon''}(m, l)).
\end{equation}
\begin{equation}\label{eq:g-spl-5}
    m = \Ginv_{\epsilon}(L_X^{\epsilon'}(i, j), L_Y^{\epsilon''}(m, l)).
\end{equation}
\begin{equation}\label{eq:g-spl-6}
    m = \Ginv_{\epsilon}(L_X^{\epsilon'}(m, j), L_Y^{\epsilon''}(k, l)).
\end{equation}
\begin{equation}\label{eq:g-spl-7}
    m = \Ginv_{\epsilon}(L_X^{\epsilon'}(i, j), L_Y^{\epsilon''}(k, m)).
\end{equation}
\begin{equation}\label{eq:g-spl-8}
    l = \Ginv_{\epsilon}(L_X^{\epsilon'}(i, j), L_Y^{\epsilon''}(k, l)).
\end{equation}
\begin{equation}\label{eq:g-spl-9}
    l = \Ginv_{\epsilon}(L_X^{\epsilon'}(i, j), L_Y^{\epsilon''}(j, l)).
\end{equation}
\begin{equation}\label{eq:g-spl-10}
    m = \Ginv_{\epsilon}(L_X^{\epsilon'}(i, j), L_Y^{\epsilon''}(j, l)).
\end{equation}
\begin{equation}\label{eq:g-spl-11}
    m = \Ginv_{\epsilon}(L_X^{\epsilon'}(m, j), L_Y^{\epsilon''}(j, l)).
\end{equation}
\begin{equation}\label{eq:g-spl-12}
    m = \Ginv_{\epsilon}(L_X^{\epsilon'}(i, m), L_Y^{\epsilon''}(k, l)).
\end{equation}
Assume that Equations \eqref{eq:g-spl-1} and \eqref{eq:g-spl-2} (resp. Equations \eqref{eq:g-spl-8} and \eqref{eq:g-spl-9}) constrain $j$ (resp. $l$) to take an at most bounded number of values when the rest of variables are kept fixed. Assume also that each one among Equations \eqref{eq:g-spl-3}-\eqref{eq:g-spl-7} and \eqref{eq:g-spl-10}-\eqref{eq:g-spl-12} constrain each one among the variables $j$, $l$ and $m$ to take an at most bounded number of values when the rest of variables are kept fixed.
\end{assumption}

\begin{lemma}[Compatibility of general link-functions]\label{lem:comp-gen}
	Suppose that Assumptions~\ref{assm:propB} and \ref{assm:genlink2} hold. Then the link-functions $L_X$ and $L_Y$ are compatible.
\end{lemma}

\begin{proof}
We proceed as in the proof of Lemma~\ref{lem:compatibility}. When the primary elements are not the same in $\pi$ and $\pi'$, one can use the same argument as in the proof of Lemma~\ref{lem:compatibility}. Otherwise, we begin with \eqref{eq:gen0}. As before, let us consider the case $r' < r$. The other case can be tackled similarly. In Case I, i.e. when $r = s - 1$, \eqref{eq:gen0} leads to the relation
	\begin{equation}\label{eq:gen1}
        i_{r + 1} = \Ginv_{\epsilon_s}(L_X^{\epsilon_r}(i_r, i_{r + 1}), L_Y^{\epsilon_{r'}}(i_{r'}, i_{r' + 1})),
	\end{equation}
    which gives a killing constraint by Assumption~\ref{assm:genlink2} (more specifically, via \eqref{eq:g-spl-1} in general, and via \eqref{eq:g-spl-2} in the degenerate case $r'+ 1 = r$). Similarly, in Case II, i.e. when $r < s - 1$, we are led to relations of the form
	\begin{equation}\label{eq:gen2}
        i_{r_m + 1} = \Ginv_{\epsilon_s}(L_X^{\epsilon_r}(i_r, i_{r + 1}), L_Y^{\epsilon_{r'}}(i_{r'}, i_{r' + 1})),
	\end{equation}
	or
	\begin{equation}\label{eq:gen3}
        i_{r_m} = \Ginv_{\epsilon_s}(L_X^{\epsilon_r}(i_r, i_{r + 1}), L_Y^{\epsilon_{r'}}(i_{r'}, i_{r' + 1})),
	\end{equation}
    where $r_m$ is a primary element with $r_m \le r + 1$. Consider \eqref{eq:gen2} first. When all the indices involved are distinct, we get killing constraints by Lemma~\ref{lem:g-gen}. The degenerate cases where some of the indices are equal are covered by Assumption~\ref{assm:genlink2}: the case $r' + 1 = r, r_m + 1 \ne r'$ is covered by \eqref{eq:g-spl-3}, the case $r' + 1 = r, r_m + 1 = r'$ is covered by \eqref{eq:g-spl-4}, the case $r' + 1 < r, r_m + 1 = r'$ is covered by \eqref{eq:g-spl-5}, and the case $r_m + 1 = r$ is covered by \eqref{eq:g-spl-6}.

    Now consider \eqref{eq:gen3}. If all the indices involved are distinct then we get killing constraints by Lemma~\ref{lem:g-gen}. As for the degenerate cases, Assumption~\ref{assm:genlink2} again comes to our aid. The case $r_m = r + 1$ gives us \eqref{eq:gen1} which is a killing constraint as seen earlier. On the other hand, the case $r_m = r' + 1$ is covered by \eqref{eq:g-spl-7}. 

    Thus we end up with killing constraints in all possible scenarios and hence the proof is complete.
\end{proof}
Together, Lemmas~\ref{lem:suff} and \ref{lem:comp-gen} imply our main theorem for general link-functions.
\begin{theorem}[General link-functions]\label{thm:gen}
    Suppose that Assumptions~\ref{assm:entries-sms}, \ref{assm:propB}, and \ref{assm:genlink2} hold. Then $n^{-1/2}X_n \odot Y_n$ converges in $*$-distribution to a circular variable in expected normalised trace.
\end{theorem}
The following are a couple of examples where one can verify Assumptions~\ref{assm:propB} and \ref{assm:genlink2}:
\begin{enumerate}
    \item $L_X(i, j) = 2i + 7j$, $L_Y(i, j) = i + 4j$.
    \item $L_X(i, j) = i^2 + j$, $L_Y(i, j) = i + j^2$.
\end{enumerate}

An example of a pair of link-functions which do not satisfy Assumption~\ref{assm:genlink2} is $(L_X(i, j) = i - j, L_Y(i, j) = 2i - j)$. One can check that in this case $\Ginv(u, v) = (v - u, v - 2u)$. For $\epsilon = \epsilon' = \epsilon'' = 1$, the constraint \eqref{eq:g-spl-1} becomes
\[
    j = L_Y(k, l) - L_X(i, j) = (2k - l) - (i - j) = 2k - l - i + j.  
\]
As $j$ vanishes from the constraint altogether, fixing the other variables does not impose any constraints on $j$, thus violating Assumption~\ref{assm:genlink2}. This creates troubles in the proof of Lemma~\ref{lem:comp-gen}. For instance, in the case $\epsilon_r = \epsilon_{r'} = \epsilon_{s} = 1$, the relation \eqref{eq:gen1} gives us
\[
    i_{r + 1} = 2i_{r'} - i_{r' + 1} - i_{r} + i_{r + 1}, 
\]
i.e.
\begin{equation}\label{eq:ir-constraint}
    i_{r} = 2i_{r'} - i_{r' + 1}.
\end{equation}
For $r > r' + 1$, unless $i_r$ is generating, we do not necessarily have a killing constraint here, and additional work is required to find one. In this case, one may argue as follows: by Lemma~\ref{lem:non-gen}, considering the constraints coming from $L_X$ only, one can determine a non-generating $i_r$ up to a bounded number of choices from its predecessor generating indices (in fact, due to the linear nature of the constraints, one can explicitly solve for $i_r$ in terms of the previous generating indices). On the other hand, \eqref{eq:ir-constraint} gives a \emph{different} linear constraint on $i_r$. Together these two separate constraints on $i_r$ create a killing constraint between the predecessor generating indices of $i_r$.

We believe that the line of argument given above works for all pairs of linear link-functions $(L_X(i, j) = ai + bj + e, L_Y(i, j) = ci + dj + f)$, where $a, b, c, d, e, f$ are integers with $a, b, c, d$ nonzero and $ad \ne bc$, even if these link-functions do not satisfy some parts of Assumption~\ref{assm:genlink2}, and the arguments given in the proof of Lemma~\ref{lem:comp-gen} do not yield killing constraints in all of the cases considered. However, since writing down the constraints explicitly for general $a, b, c, d, e, f$ is rather cumbersome, we remain content with an informal discussion.  

\subsection{Almost sure convergence}\label{sec:main-asconv}
Earlier we have shown convergence of the $*$-moments in expected normalised trace. However, a stronger result can be established under Assumption~\ref{assm:entries}---the $*$-moments in normalised trace converge almost surely to the $*$-moments of a circular variable. The proof of this is essentially the same as the proof of Lemma B.1 of \cite{bose2014bulk}, modulo some minor details. On the other hand, under Assumption~\ref{assm:entries-sms}, we are only able to show in-probability convergence. This requires some new combinatorial arguments.

\begin{theorem}\label{thm:asconv}
    Suppose that Assumption~\ref{assm:propB} holds and the link-functions $L_X$ and $L_Y$ are compatible and jointly circular. Let $c$ be a circular variable in some $*$-probability space $(\A, \state)$.
    \begin{enumerate}
        \item[(a)] If the input random variables satisfy Assumption~\ref{assm:entries-sms}, then
            \[
                \frac{1}{n}\tr\, Q(n^{-1/2} M_n, n^{-1/2} M_n^*) \xrightarrow{P} \state(Q(c, c^*)) \text{ for any } Q \in \C\langle X, X^* \rangle.
            \]
        \item[(b)] If the input variables satisfy Assumption~\ref{assm:entries}, then almost surely,             \[
                \frac{1}{n}\tr\, Q(n^{-1/2} M_n, n^{-1/2} M_n^*) \rightarrow \state(Q(c, c^*)) \text{ for any } Q \in \C\langle X, X^* \rangle.
            \]
            As a consequence, almost surely, $n^{-1/2}M_n$ converges in $*$-distribution, as an element of the $*$-probability space $(\mathcal{M}_n(\C), \frac{1}{n}\tr)$, to a circular variable.
    \end{enumerate}
\end{theorem}

Note that in order to prove almost sure/in-probability convergence of $\frac{1}{n}\tr\, Q(n^{-1/2} M_n, n^{-1/2} M_n^*)$ to $\state(Q(c, c^*))$, where $Q \in \C\langle X, X^* \rangle$, it is enough to show convergence for monomials. 
Now, since we have already shown that
\[
    \E \frac{1}{n} \tr (n^{-1/2}M_n)^{\epsilon_1} \cdots (n^{-1/2}M_n)^{\epsilon_k} \rightarrow \state(c^{\epsilon_1} \cdots c^{\epsilon_k})
\]
for any $k \ge 1$ and $\epsilon_1, \ldots, \epsilon_k \in \{1, *\}$, in order to prove almost sure/in-probability convergence, it is enough to control centred moments of $\frac{1}{n} \tr (n^{-1/2}M_n)^{\epsilon_1} \cdots (n^{-1/2}M_n)^{\epsilon_k}$ suitably. For example, in case of in-probability convergence, it is enough to show that
\begin{equation}\label{eq:vanishing_second_moment}
    \E \bigg( \frac{1}{n} \tr (n^{-1/2}M_n)^{\epsilon_1} \cdots (n^{-1/2}M_n)^{\epsilon_k} - \E \frac{1}{n} \tr (n^{-1/2}M_n)^{\epsilon_1} \cdots (n^{-1/2}M_n)^{\epsilon_k} \bigg)^2 = o(1).
\end{equation}
On the other hand, for almost sure convergence, it is sufficient to show that
\[
    \sum_{n \ge 1} \E \bigg( \frac{1}{n} \tr (n^{-1/2}M_n)^{\epsilon_1} \cdots (n^{-1/2}M_n)^{\epsilon_k} - \E \frac{1}{n} \tr (n^{-1/2}M_n)^{\epsilon_1} \cdots (n^{-1/2}M_n)^{\epsilon_k} \bigg)^4 < \infty
\]
for any $k \ge 1$ and $\epsilon_1, \ldots, \epsilon_k \in \{1, *\}$, which can be established by showing that
\[
    \E \bigg( \frac{1}{n} \tr (n^{-1/2}M_n)^{\epsilon_1} \cdots (n^{-1/2}M_n)^{\epsilon_k} - \E \frac{1}{n} \tr (n^{-1/2}M_n)^{\epsilon_1} \cdots (n^{-1/2}M_n)^{\epsilon_k} \bigg)^4 = O(n^{-(1 + \delta)})
\]
for some $\delta > 0$.

To this end, we will establish bounds on the fourth centred moment of $\frac{1}{n} \tr (n^{-1/2}M_n)^{\epsilon_1} \cdots (n^{-1/2}M_n)^{\epsilon_k}$ under Assumptions~\ref{assm:entries} and \ref{assm:entries-sms}.
\begin{lemma}\label{lem:fourth_moment_bound}
    Suppose that Assumption~\ref{assm:propB} holds and the link-functions $L_X$ and $L_Y$ are compatible and jointly circular. Let $k \ge 1$ and $\epsilon_1, \ldots, \epsilon_k \in \{1, *\}$.
    \begin{enumerate}
        \item[(a)] If the input random variables satisfy Assumption~\ref{assm:entries-sms}, then
        \[
            \E \bigg( \frac{1}{n} \tr (n^{-1/2}M_n)^{\epsilon_1} \cdots (n^{-1/2}M_n)^{\epsilon_k} - \E \frac{1}{n} \tr (n^{-1/2}M_n)^{\epsilon_1} \cdots (n^{-1/2}M_n)^{\epsilon_k} \bigg)^4 = O(n^{-1}).
        \]
        \item[(b)] If the input variables satisfy Assumption~\ref{assm:entries}, then
        \[
            \E \bigg( \frac{1}{n} \tr (n^{-1/2}M_n)^{\epsilon_1} \cdots (n^{-1/2}M_n)^{\epsilon_k} - \E \frac{1}{n} \tr (n^{-1/2}M_n)^{\epsilon_1} \cdots (n^{-1/2}M_n)^{\epsilon_k} \bigg)^4 = O(n^{-(1 + \delta)})
        \]
        for some $\delta > 0$.
    \end{enumerate}
\end{lemma}
\begin{remark}
For proving in-probability convergence under Assumption~\ref{assm:entries-sms}, it is simpler to directly show \eqref{eq:vanishing_second_moment}, but we choose to control the fourth moment instead because our arguments illustrate where we get the extra $\delta$ under Assumption~\ref{assm:entries}.
\end{remark}

Before we prove Lemma~\ref{lem:fourth_moment_bound}, we need to discuss the concept of matching of multiple multi-indices. Let $\bi_\ell$ denote the multi-index $(i^{(\ell)}_1, \ldots, i^{(\ell)}_k)$. Fix $\epsilon_1, \ldots, \epsilon_k$. We say that the multi-indices $(\bi_1, \ldots, \bi_t)$ are jointly $L$-matched if every $L$-value $L^{\epsilon_j}(i^{(\ell)}_j, i^{(\ell)}_{j + 1})$ occurs at least twice across all the multi-indices. If every $L$-value occurs exactly twice across all the multi-indices, then we say that the multi-indices are jointly $L$-matched in pairs.  We say that $(\bi_1, \ldots, \bi_t)$ are across $L$-matched, if every multi-index $\bi_\ell$ has at least one $L$-value that appears in at least one of the other multi-indices.

\begin{proof}[Proof of Lemma~\ref{lem:fourth_moment_bound}]
As in \eqref{eq:expansion}, we can use multi-indices to write
\[
    \frac{1}{n} \tr (n^{-1/2}M_n)^{\epsilon_1} \cdots (n^{-1/2}M_n)^{\epsilon_k} - \E \frac{1}{n} \tr (n^{-1/2}M_n)^{\epsilon_1} \cdots (n^{-1/2}M_n)^{\epsilon_k} = \frac{1}{n^{1 + \frac{k}{2}}}\sum_{\bi} (x_{\bi}^{\epsilon} y_{\bi}^{\epsilon} - \E x_{\bi}^{\epsilon} \E y_{\bi}^{\epsilon}).
\]
Therefore
\begin{align} \nonumber
    \E \bigg( \frac{1}{n} \tr (n^{-1/2}M_n)^{\epsilon_1} \cdots (n^{-1/2}M_n)^{\epsilon_k} - &\E \frac{1}{n} \tr (n^{-1/2}M_n)^{\epsilon_1} \cdots (n^{-1/2}M_n)^{\epsilon_k} \bigg)^4 \\ \nonumber
                                                                                             &= \frac{1}{n^{2k + 4}}\E \sum_{\bi_1, \bi_2, \bi_3, \bi_4 \in [n]^k} \prod_{\ell = 1}^4  (x_{\bi_\ell}^{\epsilon} y_{\bi_\ell}^{\epsilon} - \E x_{\bi_\ell}^{\epsilon} \E y_{\bi_\ell}^{\epsilon}) \\ \label{eq:fourth-moment}
                                                                                             &= \frac{1}{n^{2k + 4}}\sum_{\bi_1, \bi_2, \bi_3, \bi_4 \in [n]^k} \E \prod_{\ell = 1}^4  (x_{\bi_\ell}^{\epsilon} y_{\bi_\ell}^{\epsilon} - \E x_{\bi_\ell}^{\epsilon} \E y_{\bi_\ell}^{\epsilon}).
\end{align}
The strategy now is to split the sum $\sum_{\bi_1, \bi_2, \bi_3, \bi_4 \in [n]^k} \E \prod_{\ell = 1}^4  (x_{\bi_\ell}^{\epsilon} y_{\bi_\ell}^{\epsilon} - \E x_{\bi_\ell}^{\epsilon} \E y_{\bi_\ell}^{\epsilon})$ into a bounded number of sub-sums according to the matching properties of the multi-indices $(\bi_1, \bi_2, \bi_3, \bi_4)$, and show that each of these sub-sums are $O(n^{2k + 3 - \delta})$, where $\delta \ge 0$ (we show that under Assumption~\ref{assm:entries}, one can take $\delta > 0$ for all of these sub-sums, whereas under Assumption~\ref{assm:entries-sms}, $\delta = 0$ for certain sub-sums). This is done by showing that (i) the summands $\E \prod_{\ell = 1}^4  (x_{\bi_\ell}^{\epsilon} y_{\bi_\ell}^{\epsilon} - \E x_{\bi_\ell}^{\epsilon} \E y_{\bi_\ell}^{\epsilon})$ are exactly $0$ for certain sub-sums, rendering those sub-sums exactly $0$, and (ii) for the remaining sub-sums, the summands $\E \prod_{\ell = 1}^4  (x_{\bi_\ell}^{\epsilon} y_{\bi_\ell}^{\epsilon} - \E x_{\bi_\ell}^{\epsilon} \E y_{\bi_\ell}^{\epsilon})$ are bounded by some constant that depends only on $k$, and the number of such summands is $O(n^{2k + 3 - \delta})$.

We first carry out the above programme under Assumption~\ref{assm:entries-sms} to prove part (a).

\textbf{Case 1 ($(\bi_1, \bi_2, \bi_3, \bi_4)$ is neither jointly $L_X$-matched, nor jointly $L_Y$-matched).} Note that if $(\bi_1, \bi_2, \bi_3, \bi_4)$ is not jointly $L_X$-matched, then one of the $x_{\bi_\ell}^{\epsilon}$'s, say $x_{\bi_h}^{\epsilon}$, will have a single occurrence of an input random variable making $\E x_{\bi_h}^{\epsilon} = 0$ (by the admissibility of the SMS $\{x_{\alpha}\}_{\alpha \in \Z^d}$). This would mean that
\[
    \E \prod_{\ell = 1}^4  (x_{\bi_\ell}^{\epsilon} y_{\bi_\ell}^{\epsilon} - \E x_{\bi_\ell}^{\epsilon} \E y_{\bi_\ell}^{\epsilon}) = \E x_{\bi_h}^{\epsilon} y_{\bi_h}^{\epsilon} \prod_{\ell \ne h}  (x_{\bi_\ell}^{\epsilon} y_{\bi_\ell}^{\epsilon} - \E x_{\bi_\ell}^{\epsilon} \E y_{\bi_\ell}^{\epsilon}) = 0.
\]
Here we are using the admissibility $\{x_{\alpha}\}_{\alpha \in \Z^d}$, and the independence of $\{x_{\alpha}\}_{\alpha\in\Z^d}$ and $\{y_{\beta}\}_{\beta \in \Z^{d'}}$. The same conclusion can be made if $(\bi_1, \bi_2, \bi_3, \bi_4)$ is not jointly $L_Y$-matched.

\textbf{Case 2 ($(\bi_1, \bi_2, \bi_3, \bi_4)$ is both jointly $L_X$-matched and jointly $L_Y$-matched, but neither across $L_X$-matched, nor across $L_Y$-matched).} In this case, some multi-index, say $\bi_{h_1}$, is only self-$L_X$-matched, i.e. none of its $L_X$-values appear among the $L_X$-values of any of the other multi-indices. Similarly, some multi-index, say $\bi_{h_2}$ is only self-$L_Y$-matched. The case $h_1 \ne h_2$ can be tackled in the exact same manner as in the proof of Lemma~B.1 of \cite{bose2014bulk}. The idea is that by Assumption~\ref{assm:entries-sms}, the entry variables have uniformly bounded moments of all orders, so by an application of H\"{o}lder's inequality, $\E\prod_{\ell = 1}^4 (x_{\bi_\ell}^{\epsilon} y_{\bi_\ell}^{\epsilon} - \E x_{\bi_\ell}^{\epsilon} \E y_{\bi_\ell}^{\epsilon})$ is bounded by some constant that depends only on $k$, and not $n$. Thus it is enough to bound the possible number of quadruples of multi-indices $(\bi_1, \bi_2, \bi_3, \bi_4)$ obeying the condition $h_1 \ne h_2$. In \cite{bose2014bulk}, an $O(n^{2k + 3 - \delta})$ bound for some $\delta > 0$ is established via a case-by-case counting argument.

Now consider the case $h_1 = h_2$. Without loss of generality, we assume that $h_1 = h_2 = 1$. Then $\bi_1$ can be chosen in $O(n^{1 + \lfloor\frac{k}{2}\rfloor})$ ways. The remaining multi-indices are jointly $L_X$-matched as well as jointly $L_Y$-matched. We now employ a case-by-case analysis. 
\begin{enumerate}
    \item Consider the case where $(\bi_2, \bi_3, \bi_4)$ is either across $L_X$-matched or across $L_Y$-matched. In this case, Lemma~A.1 of \cite{bose2014bulk} tells us that $(\bi_2, \bi_3, \bi_4)$ can be chosen in $O(n^{\lfloor \frac{3k}{2}\rfloor + 2})$ ways. Thus the total number of choices for $(\bi_1, \bi_2, \bi_3, \bi_4)$ is $O(n^{\lfloor\frac{k}{2}\rfloor + \lfloor \frac{3k}{2}\rfloor + 3})$ which is $O(n^{2k + 3})$ when $k$ is even. In fact, by considering the case where the multi-indices are jointly $L_X$-matched in pairs as well as jointly $L_Y$-matched in pairs, it is not hard to see that the number of choices is exactly of order $n^{2k + 3}$. Thus, in this case, we cannot get a $O(n^{2k + 3 - \delta})$ bound for some $\delta > 0$.

    \item Now consider the case where $(\bi_2, \bi_3, \bi_4)$ is neither across $L_X$-matched nor across $L_Y$-matched. Then some $\bi_{h_3}$ is only self-$L_X$-matched, and some $\bi_{h_4}$ is only self-$L_Y$-matched. Again, we may have two contingencies: $h_3 = h_4$ or $h_3 \ne h_4$. If $h_3 \ne h_4$, then we can proceed like the $h_1 \ne h_2$ case and ultimately prove an $O(n^{2k + 3 - \delta})$ bound. So let us consider the case $h_3 = h_4$. Assume, without loss of generality, that $h_3 = h_4 = 2$. Then $\bi_2$ can be chosen in $O(n^{1 + \lfloor\frac{k}{2}\rfloor})$ ways. 
        \begin{enumerate}
            \item[i.] If $(\bi_3, \bi_4)$ is either across $L_X$-matched or across $L_Y$-matched, then by Lemma~A.1 of \cite{bose2014bulk}, $(\bi_2, \bi_3)$ can be chosen in $O(n^{k + 1})$ ways. Thus, in this case, the total number of choices for  $(\bi_1, \bi_2, \bi_3, \bi_4)$ is
                \[
                    \underbrace{O(n^{1 + \lfloor\frac{k}{2}\rfloor})}_{\bi_1} \times \underbrace{O(n^{1 + \lfloor\frac{k}{2}\rfloor})}_{\bi_2} \times \underbrace{O(n^{k + 1})}_{(\bi_3, \bi_4)} = O(n^{2k + 3}).
                \]
                Again, the upper bound can be attained by considering multi-indices that are jointly $L_X$-matched in pairs as well as jointly $L_Y$-matched in pairs.

            \item[ii.] Otherwise, $(\bi_3, \bi_4)$ is neither across $L_X$-matched nor across $L_Y$-matched. In this case, all the multi-indices are self-$L_X$-matched as well as self-$L_Y$-matched. If any one of these multi-indices has an $L_X$-value or $L_Y$-value repeated at least thrice, then that multi-index can be chosen in at most $O(n^{\lfloor \frac{k + 1}{2} \rfloor})$ ways (this bound can also be attained). In that case, the total number of choices for  $(\bi_1, \bi_2, \bi_3, \bi_4)$ is 
                \[
                    O(n^{\lfloor \frac{k + 1}{2} \rfloor}) \times O((n^{1 + \lfloor\frac{k}{2}\rfloor})^3) = O(n^{2k + 3}).
                \]
                However, if each of the multi-indices is self-matched in pairs with respect to both $L_X$ and $L_Y$, then the total number of choices is $O((n^{1 + \lfloor\frac{k}{2}\rfloor})^4)= O(n^{2k + 4})$. This would have led to a $O(1)$ bound on \eqref{eq:fourth-moment}, were it not for the fact that we can do an exact computation in this case. As each of the multi-indices is self-matched in pairs with respect to both $L_X$ and $L_Y$, we get by strong multiplicativity that $\E x_{\bi_\ell}^\epsilon = 1 = \E x_{\bi_\ell}^\epsilon$ for any $\ell$. This means that
                \[
                    \E \prod_{\ell = 1}^4  (x_{\bi_\ell}^{\epsilon} y_{\bi_\ell}^{\epsilon} - \E x_{\bi_\ell}^{\epsilon} \E y_{\bi_\ell}^{\epsilon}) = \E \prod_{\ell = 1}^4  (x_{\bi_\ell}^{\epsilon} y_{\bi_\ell}^{\epsilon} - 1).
                \]
                It is easy to see using strong multiplicativity and the independence of the $x_{\alpha}$'s and $y_{\beta}$'s that the product on the right-hand side vanishes when we expand it.
        \end{enumerate}
\end{enumerate}
We thus see that under Assumption~\ref{assm:entries-sms}, we only have an $O(n^{2k + 3})$ upper bound on the number of multi-indices in the $h_1 = h_2$ case.

\textbf{Case 3 ($(\bi_1, \bi_2, \bi_3, \bi_4)$ is both jointly $L_X$-matched and jointly $L_Y$-matched, and also either across $L_X$-matched or across $L_Y$-matched).} In this case, by Equation (B.1) of \cite{bose2014bulk}, we get that $(\bi_1, \bi_2, \bi_3, \bi_4)$ can be chosen in $O(n^{2k + 2})$ ways.

All in all, we have established that
\[
    \sum_{\bi_1, \bi_2, \bi_3, \bi_4 \in [n]^k} \E \prod_{\ell = 1}^4  (x_{\bi_\ell}^{\epsilon} y_{\bi_\ell}^{\epsilon} - \E x_{\bi_\ell}^{\epsilon} \E y_{\bi_\ell}^{\epsilon}) = O(n^{2k + 3}),
\]
which, in view of \eqref{eq:fourth-moment}, completes the proof of part (a).

We now turn to part (b). In Cases 1 and 3, and in the $h_1 \ne h_2$ sub-case of Case 2, we have already established $O(n^{2k + 3 - \delta})$ bounds, for some $\delta > 0$, under Assumption~\ref{assm:entries-sms}, which therefore continue to hold under the special case of Assumption~\ref{assm:entries}. In the $h_1 = h_2$ sub-case of Case 2, for which we only have an $O(n^{2k + 3})$ bound under Assumption~\ref{assm:entries-sms}, we can do much better under Assumption~\ref{assm:entries}. Indeed, when $h_1 = h_2$ and Assumption~\ref{assm:entries} is true, $x_{\bi_{h_1}}^{\epsilon}$ is independent of $\{x_{\bi_\ell}^{\epsilon} : \ell \ne h_1\}$, and $y_{\bi_{h_1}}^{\epsilon}$ is independent of $\{y_{\bi_\ell}^{\epsilon} : \ell \ne h_1\}$. Also, since the $x_{\alpha}$'s and the $y_{\beta}$'s are independent, we can write
\[
    \E\prod_{\ell = 1}^4 (x_{\bi_\ell}^{\epsilon} y_{\bi_\ell}^{\epsilon} - \E x_{\bi_\ell}^{\epsilon} \E y_{\bi_\ell}^{\epsilon}) = \underbrace{\E (x_{\bi_{h_1}}^{\epsilon} y_{\bi_{h_1}}^{\epsilon} - \E x_{\bi_{h_1}}^{\epsilon} \E y_{\bi_{h_1}}^{\epsilon})}_{=0} \E\prod_{\ell \ne h_1}^4 (x_{\bi_\ell}^{\epsilon} y_{\bi_\ell}^{\epsilon} - \E x_{\bi_\ell}^{\epsilon} \E y_{\bi_\ell}^{\epsilon}) = 0.
\]
Thus the contribution of the relevant sub-sum is exactly $0$. This leads us to an overall $O(n^{2k + 3 - \delta})$ bound on $\sum_{\bi_1, \bi_2, \bi_3, \bi_4 \in [n]^k} \E \prod_{\ell = 1}^4  (x_{\bi_\ell}^{\epsilon} y_{\bi_\ell}^{\epsilon} - \E x_{\bi_\ell}^{\epsilon} \E y_{\bi_\ell}^{\epsilon})$ under Assumption~\ref{assm:entries}, and completes the proof of part (b).
\end{proof}

\begin{proof}[Proof of Theorem~\ref{thm:asconv}]
As we have discussed earlier, in-probability (resp. almost sure) convergence of $\frac{1}{n}\tr\, Q(n^{-1/2} M_n, \allowbreak n^{-1/2} M_n^*)$ under Assumption~\ref{assm:entries-sms} (resp. Assumption~\ref{assm:entries}) to $\state(Q(c, c^*))$, where $Q \in \C\langle X, X^* \rangle$, follows from Lemma~\ref{lem:fourth_moment_bound}.

Further, as the collection of monomials is countable, after one shows almost sure convergence for monomials, one can club all the relevant null sets into a single null set outside which pointwise convergence holds for monomials and hence for all polynomials. This proves that almost surely, $n^{-1/2}M_n$ converges in $*$-distribution, as an element of the $*$-probability space $(\mathcal{M}_n(\C), \frac{1}{n}\tr)$, to a circular variable under Assumption~\ref{assm:entries}.
\end{proof}

From Theorem~\ref{thm:asconv}-(a), we can obtain as a direct corollary one of the main results of \cite{bose2014bulk}, namely the almost sure weak convergence of the ESM of the $n^{-1/2}$-scaled Schur-Hadamard product of symmetric Toeplitz and Hankel matrices to the semi-circular law, under Assumption~\ref{assm:entries}.
\begin{cor}\label{cor:bm14}
Suppose that Assumption~\ref{assm:entries} holds. If $L_X(i, j) = |i - j|$ and $L_Y(i, j) = i + j$, then the ESM of $n^{-1/2}M_n$ converges weakly almost surely to the semi-circular law.
\end{cor}
\begin{proof}
    If $M_n$ is as in Theorem~\ref{thm:toephank}, then $\frac{M_n + M_n^*}{\sqrt{2}}$ is the Schur-Hadamard product of a symmetric Toeplitz matrix and a Hankel matrix, whose entries satisfy Assumption~\ref{assm:entries}, except that the diagonal entries of the symmetric Toeplitz matrix has variance $2$---but this does not matter since changing the diagonal entries does not affect the LSM (see, e.g., Theorem 2.5 of \cite{bai2010spectral}). Now, using Theorem~\ref{thm:asconv}, we conclude that almost surely for all $k \ge 1$,
    \begin{equation*}
        \frac{1}{n}\tr\bigg(\frac{M_n + M_n^*}{\sqrt{2n}}\bigg)^k \xrightarrow{n \rightarrow \infty} \state\bigg(\frac{c + c^*}{\sqrt{2}}\bigg)^k.
    \end{equation*}
   
    This implies the desired result by moment method because $\frac{c + c^*}{\sqrt{2}}$ is a semi-circular variable if $c$ is circular.
\end{proof}

\section*{Acknowledgements}
We thank the anonymous referees for their many helpful comments and suggestions that significantly improved the paper.

\bibliographystyle{abbrv}
\bibliography{circ-hadamard-arxiv}

\begin{thebibliography}{10}

\bibitem{alexits1961convergence}
G.~Alexits.
\newblock {\em Convergence Problems of Orthogonal Series}.
\newblock Pergamon Press, New York-Oxford-Paris, 1961.

\bibitem{bai2010spectral}
Z.~Bai and J.~W. Silverstein.
\newblock {\em Spectral Analysis of Large Dimensional Random Matrices}.
\newblock Springer, New York, 2nd edition, 2010.

\bibitem{basak2018circular}
A.~Basak, N.~Cook, and O.~Zeitouni.
\newblock Circular law for the sum of random permutation matrices.
\newblock {\em Electronic Journal of Probability}, 23(33):1--51, 2018.

\bibitem{bose2018patterned}
A.~Bose.
\newblock {\em Patterned Random Matrices}.
\newblock CRC Press, Boca Raton, FL, 2018.

\bibitem{bose2014bulk}
A.~Bose and S.~S. Mukherjee.
\newblock Bulk behavior of {S}chur--{H}adamard products of symmetric random
  matrices.
\newblock {\em Random Matrices: Theory and Applications}, 3(02):1450007, 2014.

\bibitem{bryc2006spectral}
W.~Bryc, A.~Dembo, and T.~Jiang.
\newblock Spectral measure of large random {H}ankel, {M}arkov and {T}oeplitz
  matrices.
\newblock {\em The Annals of Probability}, 34(1):1--38, 2006.

\bibitem{gaposhkin1969central}
V.~F. Gaposhkin.
\newblock A central limit theorem for strongly multiplicative systems of
  functions.
\newblock {\em Mathematical notes of the Academy of Sciences of the USSR},
  6(4):720--724, 1969.

\bibitem{guionnet2011single}
A.~Guionnet, M.~Krishnapur, and O.~Zeitouni.
\newblock The single ring theorem.
\newblock {\em Ann. of Math. (2)}, 174(2):1189--1217, 2011.

\bibitem{hammond2005distribution}
C.~Hammond and S.~J. Miller.
\newblock Distribution of eigenvalues for the ensemble of real symmetric
  {T}oeplitz matrices.
\newblock {\em Journal of Theoretical Probability}, 18(3):537--566, jul 2005.

\bibitem{meckes2009some}
M.~W. Meckes.
\newblock Some {R}esults on {R}andom {C}irculant {M}atrices.
\newblock In {\em High dimensional probability V: the Luminy volume}, pages
  213--223. Institute of Mathematical Statistics, Beachwood, OH, 2009.

\bibitem{Mingo2017}
J.~A. Mingo and R.~Speicher.
\newblock {\em Free Probability and Random Matrices}.
\newblock Springer, New York, 2017.

\bibitem{nguyen2015elliptic}
H.~H. Nguyen and S.~O'Rourke.
\newblock The elliptic law.
\newblock {\em Int. Math. Res. Not. IMRN}, (17):7620--7689, 2015.

\bibitem{Nica2006}
A.~Nica and R.~Speicher.
\newblock {\em Lectures on the Combinatorics of Free Probability}.
\newblock Cambridge University Press, Cambridge, 2006.

\bibitem{tao2010random}
T.~Tao, V.~Vu, and M.~Krishnapur.
\newblock Random matrices: Universality of {ESD}s and the circular law.
\newblock {\em The Annals of Probability}, 38(5):2023--2065, 2010.

\bibitem{voiculescu1991limit}
D.~Voiculescu.
\newblock Limit laws for random matrices and free products.
\newblock {\em Inventiones mathematicae}, 104(1):201--220, 1991.

\end{thebibliography}

\end{document}